\documentclass[journal,10pt]{IEEEtran}
\usepackage{color}
\usepackage{paralist,xparse}
\usepackage{mathtools}
\usepackage{amsmath,epsfig,subcaption,paralist,verbatim,tikz,theoremref}
\usepackage{amsfonts}
\usepackage{amsmath,amsthm}
\usepackage{amssymb}
\usepackage{epsfig,psfrag}
\usepackage{dsfont}
\usepackage[hidelinks]{hyperref}
\usepackage{epstopdf}
\usepackage{cleveref}
\usepackage{etoolbox}
\usepackage[export]{adjustbox}
\IEEEoverridecommandlockouts
\newtheorem{definition}{Definition}[section]
\newtheorem{theorem}{Theorem}[section]
\newtheorem{corollary}{Corollary}[section]
\newtheorem{lemma}{Lemma}[section]

\newcommand\Sec{Sec.\@}
\newcommand\Fig{Fig.\@}
\newif\ifcompress
\compresstrue
\usepackage{algorithm,algpseudocode}
\algnewcommand{\Initialize}[1]{%
  \State \textbf{Initialize:}
  \Statex \hspace*{\algorithmicindent}\parbox[t]{.8\linewidth}{\raggedright #1}
}
\newcommand{\paperbibstyle}{IEEEtran}
\graphicspath{./figs}

\newcommand{\beq}{\begin{equation}}
\newcommand{\eeq}{\end{equation}}
\newcommand{\E}{\mathbf{E}}

\renewcommand{\Pr}{\mathbb{P}}

\newcommand{\argmax}{\operatorname{argmax}}
\newcommand{\argmin}{\operatorname{argmin}}

  \def\1{{\mathbf 1}}
\newcommand{\reals}{{\rm I\hspace{-.07cm}R}}

\newcommand{\p}{\prime}
\newcommand{\norm}[1]{\lVert#1\rVert}
\newcommand{\tindx}{t}
\newcommand{\sindx}{s}
\newcommand{\Tindxter}{T}

\newcommand{\probe}{p}
\newcommand{\response}{x}

\newcommand{\utility}{u}

\newcommand{\dataset}{\mathcal{D}}

\newcommand{\noise}{w}
\newcommand{\obsdataset}{\mathcal{D}_\text{obs}}
\newcommand{\obsresponse}{y}

\def\1{{\mathbf 1}}

\crefformat{equation}{(#2#1#3)}
\crefrangeformat{equation}{(#3#1#4) to~(#5#2#6)}
\crefmultiformat{equation}{(#2#1#3)}%
{ and~(#2#1#3)}{, (#2#1#3)}{ and~(#2#1#3)}

\usepackage{booktabs}

\begin{document}

%\markboth{A. Aprem and V. Krishnamurthy}{Utility Change Point Detection in Online Social Media}
\title{Utility Change Point Detection in Online Social Media: A Revealed Preference Framework\thanks{The authors are with the Dept. of ECE, The University of British Columbia, Vancouver, B.C.  Canada. Emails: \{aaprem, vikramk\}@ece.ubc.ca. \newline A significantly shortened version of some of the ideas in this paper have been accepted to ICASSP 2017.}}
\author{\IEEEauthorblockN{Anup Aprem and Vikram Krishnamurthy, \emph{Fellow, IEEE}}} %
%\thanks{TBD}
\maketitle
\begin{abstract}
	This paper deals with change detection of utility maximization behaviour in online social media. 
	%Such changes are motivated by the effect of marketing and advertising in online social media. 
	Such changes occur due to the effect of marketing, advertising, or changes in ground truth. 
	%We extend Afriat's theorem on revealed preference to the case where the utility function of the agent changes at an unknown time. 
	%Inspired by taste change models in microeconomics, we restrict the utility changes to be linear perturbation on a base utility function. 
	First, we use the revealed preference framework to detect the unknown time point (change point) at which the utility function changed.
	%We study the problem of detecting the unknown time point at which the utility function changes. 
	We derive necessary and sufficient conditions for detecting the change point. % at which the utility function changed. 
	Second, in the presence of noisy measurements, we propose a method to detect the change point and construct a decision test. 
	Also, an optimization criteria is provided to recover the linear perturbation coefficients. % in the presence of noisy measurement of response. 
	Finally, to reduce the computational cost, a dimensionality reduction algorithm using Johnson-Lindenstrauss transform is presented. 
	The results developed are illustrated on two real datasets: Yahoo! Tech Buzz dataset and Youstatanalyzer dataset. 
	By using the results developed in the paper, several useful insights can be gleaned from these data sets. 
	First, the changes in ground truth affecting the utility of the agent can be detected by utility maximization behaviour in online search. 
	Second, the recovered utility functions satisfy the single crossing property indicating strategic substitute behaviour in online search. 
	Third, due to the large number of videos in YouTube, the utility maximization behaviour was verified through the dimensionality reduction algorithm. 
	Finally, using the utility function recovered in the lower dimension, we devise an algorithm to predict total traffic in YouTube. \\ 
\end{abstract}
%\category{G.3}{Probability and Statistics}{Nonparametric statistics}
%\category{G.3}{Probability and Statistics}{Time series analysis}
%\category{J.4}{Social and Behavioral Sciences}{Economics}
%\terms{Algorithms, Economics, Theory}
%\acmformat{Anup Aprem and Vikram Krishnamurthy, 2015. Utility Change Point Detection in Online Social Media.}

%\begin{bottomstuff}
%Author's address: Department of Electrical \& Computer Engineering, Kaiser Building, 2332 Main Mall, The University of British Columbia, Vancouver, B.C., Canada, V6T 1Z4
%\end{bottomstuff}

\section{Introduction}
\label{sec:intro}
%This paper deals with the problem of non-parametric change detection of ``utility maximization'' behaviour in online social media. 
%The problem we consider is fundamentally different to the theme used widely in the signal processing literature, where one postulates an objective function (typically convex) and then develops optimization algorithms. 
%In contrast, the revealed preference approach, considered in this paper, is {\em data centric} - given a dataset, we wish to determine if is consistent with utility maximization, and then detect changes in the utility function based on the observed behaviour. 
%In recent years, social media has become ubiquitous for networking and content sharing. % Rewrite
The interaction of humans on social media platforms mimic their interactions in the real world~\cite{RN96}. 
Hence, as in the real world, ``utility maximization'' underpins human interaction on social media platforms. 
Utility maximization is the fundamental problem humans face, wherein humans maximize utility given their limited resources of money or attention. 
Detection of utility maximization behaviour is therefore useful in online social media. 
However, a key difference for agents in online social media is the absence of economic incentives. 
For example, the majority of content in Facebook, YouTube and Twitter are user-generated with limited or no economic incentives. 
Hence, as explained in~\cite{TS13}, incentives such as ``fun'' and ``fame'' are some of the major attributes of the utility function of online social behaviour. 
It is therefore difficult to analytically characterize the utility function and hence any detection of utility maximization behaviour in online social media needs to be necessarily nonparametric in nature. 
Another key difference is that the utility function in online social media is ``content-aware'', i.e.\ the quality of content affects the utility function. 
Due to the content-aware nature of the utility function and the availability of large amount of user generated content, data from online social media is high-dimensional. 
For example, utility maximization in YouTube depends on all the user generated video content available at any point of time. 

The problem of nonparametric detection of utility maximizing behaviour is the central theme in the area of revealed preferences in microeconomics. This is fundamentally different to the theme used widely in the signal processing literature, where one postulates an objective function (typically convex) and then develops optimization algorithms. In contrast, the revealed preference framework, considered in this paper, is {\em data centric}:
%given a dataset, we wish to determine if it is consistent with utility maximization. 
Given a dataset, $\dataset$, consisting of probe, $p_t \in \mathbb{R}_+^m$, and response, $x_t  \in \reals_+^m$, of an agent for $T$ time instants:  
\begin{equation}
	\dataset = \left\{(p_t,x_t), t=1,2,\dots,T\right\}. 
	\label{eqn:dataset}
\end{equation}
Revealed preference aims to answer the following question: 
Is the dataset $\dataset$ in~\eqref{eqn:dataset} consistent with utility-maximization behaviour of an agent? % in~\eqref{eqn:utilitymaximization}?
A utility-maximization behaviour (or utility maximizer) is defined as follows:
\begin{definition} % Include reference
	An agent is a utility maximizer if, at each time $t$, for input probe $p_t$, the output response, $x_t$, satisfies 
	\begin{equation}
		x_t = x(p_t) \in \underset{\left\{p_t^\prime x \le I_t\right\}}{\argmax u(x)}.
		\label{eqn:utilitymaximization}
	\end{equation}
	%where $p_t, x_t \in \mathbb{R}_+^m$ and $u(x)$ denotes a locally non-satiated\footnote{Local non-satiation means that for any point, $x$, there exists another point, $y$, within an $\varepsilon$ distance (i.e. $\norm{x-y}\le \varepsilon$), such that the point $y$ provides a higher utility than $x$ (i.e. $u(x)<u(y)$).} utility function. Also, $I_t\in\mathbb{R}_+$, is the total resource or the budget of the agent. 
	Here, $u(x)$ denotes a locally non-satiated\footnote{Local non-satiation means that for any point, $x$, there exists another point, $y$, within an $\varepsilon$ distance (i.e. $\norm{x-y}\le \varepsilon$), such that the point $y$ provides a higher utility than $x$ (i.e. $u(x)<u(y)$). Local non-satiation models the human preference: \emph{more is preferred to less}} utility function\footnotemark. Also, $I_t\in\mathbb{R}_+$, is the budget of the agent. The linear constraint, $p_t^\prime x \le I_t$ imposes a \emph{budget constraint} on the agent, where $p_t^\prime x$ denotes the inner product between $p_t$ and $x$. 
\end{definition}
\footnotetext{The utility function is a function that captures the preference of the agent. For example, if $x$ is preferred to $y$, $u(x) \ge u(y)$. }
%Note that the classical theory of revealed preference deals only with the case where the utility function of the agent is static, i.e. the utility function does not change with time. 
Major contributions to the area of revealed preferences are due to Samuelson~\cite{Samuelson38}, Afriat~\cite{Afr67}, Varian~\cite{Var82}, and Diewert~\cite{Die73} in the microeconomics literature. 
Afriat~\cite{Afr67} devised a nonparametric test (called Afriat's theorem), which provides necessary and sufficient conditions to detect utility maximizing behaviour for a dataset. 
For an agent satisfying utility maximization, Afriat's theorem~\cite{Afr67} (stated below in Theorem~\ref{thm:AfriatTheorem}) provides a method to reconstruct a utility function consistent with the data. 
The utility function, so obtained, can be used to predict future response of the agent. 
Varian~\cite{Var06} provides a comprehensive survey of revealed preference literature. 

Despite being originally developed in economics, there has been some recent work on application of revealed preference to social networks and signal processing. 
In the signal processing literature, revealed preference framework was used for detection of malicious nodes in a social network in~\cite{KH12,BP13} and in demand estimation in smart grids in~\cite{HK15}. 
\cite{CJP10} analyzes social behaviour and friendship formation using revealed preference among high school friends. In online social networks, \cite{JPW07} uses revealed preference to obtain information about products from bidding behaviour in eBay or similar bidding networks. 
%Similarly, \cite{PS07} uses revealed preference to understand social influence on travel behaviour. Revealed preference is used in the online search market to estimate demand for hotel bookings in~\cite{Sergi10}. 
%In finance, revealed preference techniques are used for asset pricing~\cite{BB16} and bank mergers~\cite{ACH15}. 
\subsection{The Problem: Utility Change Point Detection in Online Social Media.}
%In this paper, we study the utility maximization process of agent with ``dynamic utility functions'' under the revealed preference setting. 
In this paper, we consider an extension of the classical revealed preference framework of~\cite{Afr67} to agents with ``dynamic utility function''. The utility function jump changes at an unknown time instant by a linear perturbation. Given the dataset of probe and responses of an agent, the objective is to develop a nonparametric test to detect the change point and the utility functions before and after the change, which is henceforth referred to as the change point detection problem. 

Such change point detection problems arise in online search in social media. 
The online search is currently the most popular method for information retrieval~\cite{SEW04}. 
The online search process can be seen as an example of an agent maximizing the information utility, i.e.\ the amount of information consumed by an online agent given the limited resource on time and attention. 
There has been a gamut of research which links internet search behaviour to ground truths such as symptoms of illness, political election, or major sporting events~\cite{GMPBSMB09,CM09,ZYF11,SSGA10,Doo10,WB09}. 
%Hence, a change in the utility in the online search corresponds to change in ground truth or exogenous events affecting the utility of agent, such as the onset of disease or the announcement of major political decision. 
Detection of utility change in online search, therefore, is helpful to identify changes in ground truth and useful, for example, for early containment of diseases~\cite{CM09} or predicting changes in political opinion~\cite{TSSW10,KJ02}. 
Also, the intrinsic nature of the online search utility function motivates such a study under a revealed preference setting. 
%The term ``dynamic utility function'' is used in the following restrictive sense: The agent is assumed to be a utility maximizer, however, because of exogenous events, affecting the agent, the utility function jump changes at an unknown time point. 
%Hence, for an agent with a dynamic utility function, a problem of particular interest is to find the unknown time point at which the utility function changed, which is henceforth referred to as the change point detection problem. 

%In this paper, we restrict our attention to determine a single change point\footnotemark at which the utility of the agent changed. \footnotetext{The extension to multiple change point is straightforward. }
%In addition, we restrict our attention to a linear perturbation in the utility function and the objective is to develop a nonparametric test to detect the unknown time at which the utility function of the agent changed. 
%The choice of linear perturbation is motivated by several reasons. 
The problem of studying agents with dynamic utility functions, with a linear perturbation change in the utility function is motivated by several reasons. 
First, it provides sufficient selectivity such that the non-parametric test is not trivially satisfied by all datasets but still provides enough degrees of freedom. 
Second, the linear perturbation can be interpreted as the change in the marginal rate of utility relative to a ``base'' utility function. 
In online social media, the linear perturbation coefficients measure the impact of marketing or the measure of severity of the change in ground truth on the utility of the agent. 
This is similar to the linear perturbation models used to model taste changes~\cite{ABBC15,MF12,BrMa98,FIS15} in microeconomics. 
Finally, in social networks, linear change in the utility is usually used to model the change in utility of an agent based on the interaction with the agent's neighbours~\cite{CS10}. 
Compared to the taste change model, our model is unique in that we allow the linear perturbation to be introduced at an unknown time. 
To the best of our knowledge, this is the first time in the literature that change point detection problem has been studied in the revealed preference setting. 

A related important practical issue that we also consider in this paper is the high dimensionality of data arising in online social media. 
As an example of high dimensional data arising in online social media, we investigate the detection of the utility maximization process inherent in one of the most common social media interactions: video sharing via YouTube\footnote{YouTube has millions of videos.\url{https://www.youtube.com/yt/press/statistics.html}}. 
%The utility function of YouTube is content-aware as it depends on the quality of the video content available at any point of time. 
%The objective is to analyze the user behaviour based on the quality of the content in YouTube. 
%We approximate the quality of video content using the two measurable quantities: the number of subscribers and the number of views of the video. 
%It is well known that the number of views of a YouTube video increases with the number of subscribers~\cite{KS85,SOT12}. 
%To determine the utility maximization in YouTube, the probe that we choose is the inverse of the number of subscribers and the response is the number of views for each video. 
%The inner product of the probe and the response vector gives the sum of the ``view focus'', as defined in~\cite{BSW12} of all the videos. 
%If we restrict our attention to popular videos, view focus tend to remain constant during a given time period and this corresponds to the linear budget constraint, $p_t^\prime x \le I_t$, in the revealed preference setting. 
%However, note that the dimension of the probe and response vector, in this case, is large, as YouTube have millions of videos. 
Detecting utility maximization behaviour with such high dimensional (big) data is computationally demanding. 

%\subsection{Organization}
The organization of the paper is as follows:
In \Sec~\ref{sec:ucpd}, we derive necessary and sufficient conditions for change point detection, for dynamic utility maximizing agents under the revealed preference framework. 
%Also, we provide an optimization criteria to recover the linear perturbation and the utility function. 
%In the presence of noise, the non-parametric test developed above may not be satisfied. 
In \Sec~\ref{sec:recnoise}, we study the change point detection problem in the presence of noise. 
%Under the classical revealed preference setting\cite{Afr67} (when the utility function do not change with time), when the dataset is measured in noise, the failure of Afriat's theorem could either be due to utility maximization not being satisfied or due to measurement errors. 
%One approach is to construct a goodness-of-fit measure for the Afriat's theorem~\cite{Var06,SSCD14}. 
%The alternative approach is to construct a statistical test such as in~\cite{FW05} and~\cite{KH12}. 
%\cite{KH12} proposes a statistical test with guaranteed false alarm probability to detect utility maximization in the presence of noise.  
%However, a Monte Carlo simulation is required to compute the threshold of the statistical test. 
%We first derive analytical expressions for the lower bound of false alarm probability in \Sec~\ref{subsec:classicalrevpref}, for the classical revealed preference setting. 
%The techniques presented in \Sec~\ref{subsec:classicalrevpref} are extended to dynamic utility maximizing agents in \Sec~\ref{subsec:recover:linear:perturbation}. 
%In \Sec~\ref{subsubsec:detect:changepoint:noise}, we propose a method to detect the unknown change point in the presence of noise. 
%Also, an optimization criteria to recover the linear perturbation coefficients to minimize the false alarm probability is derived in \Sec~\ref{subsubsec:recover:linear:perturbation}. 
Section~\ref{sec:dimensionreduction} address the problem of high dimensional data arising in the context of revealed preference. 
%In the context of high dimensional data in utility maximization behaviour, this paper proposes the technique of dimensionality reduction in revealed preference. 
%The high dimensional data is projected into a lower dimensional subspace using the Johnson-Lindenstrauss (JL) transform. 
%The dimensionality reduction techniques, that we present, can be applied both to the classical revealed preference setting as well as for agents with dynamic utility function. 
%Also, the dimensionality reduction serve as a useful visualization tool. 
%The non-parametric tests can then be applied to the lower dimensional data, thus overcoming the computational cost associated with the high dimensional data. 
Section~\ref{sec:results} presents numerical results. First, we compare the proposed approach with the popular CUSUM test and corresponding ROC curves are presented. 
Second, we illustrate the result developed on two real world datasets: Yahoo! Tech Buzz dataset and Youstatanalyzer dataset. 
%The theory developed gives new insight into utility maximizing behaviour in online social media. 
%First, the utility maximization in online search is affected by ground truths such as the announcement of new products by a major technology company. 
%The recovered utility function satisfy single crossing property indicating strategic substitute behaviour in online search behaviour. 
%Second, we find that video sharing in YouTube satisfies utility maximization (satisfies Afriat's theorem).  
%The paper also provides an algorithm for predicting the total incoming traffic (number of views) to YouTube, based on the utility function obtained in the lower dimension. 
%Concluding remarks are offered in Section~\ref{sec:conclusion}. 
\section{Utility Change Point Detection (Deterministic case)}
\label{sec:ucpd}
In this section, we consider agents with a dynamic utility function. 
For completeness, we start with Afriat's theorem, in the classical static setting. 
Afriat's Theorem\footnotemark is one of the \emph{big} results in revealed preferences in micro-economics theory. \footnotetext{To the signal processing reader unfamiliar with this theorem, it can be viewed as a set-valued system identification method for an $\argmax$ nonlinear system with a constraint on the inner product of the probe and response of a system. Afriat's theorem has several interesting consequences including the fact that if a dataset is consistent with utility maximization, then it is rationalizable by a concave, monotone and continuous utility function. Hence, the preference of the agent represented by a concave utility function can never be refuted based on a finite dataset, see~\cite{Var06}. Further, we can always impose monotone and concave restrictions on the utility function with no loss of generality. }
\begin{theorem}[Afriat's Theorem~\cite{Afr67}]{Given the dataset $\mathcal{D}$ in~\eqref{eqn:dataset}, the following statements are equivalent:}
	\begin{compactenum}
	\item The agent is a utility maximizer and there exists a monotonically increasing\footnote{In this paper, we use monotone and local non-satiation interchangeably. Afriat's theorem was originally stated for a non-satiated utility function.} and concave\footnote{Concavity of utility function models the human preference: \emph{averages are better than the extremes}. It is also related to the law of diminishing marginal utility, i.e.\ the rate of utility decreases with $x$. }utility function that satisfies (\ref{eqn:utilitymaximization}). 
		\item For $u_t$ and $\lambda_t>0$ the following set of inequalities has a feasible solution:
			\begin{equation}
				\utility_\sindx-\utility_\tindx-\lambda_\tindx \probe_\tindx^\p (\response_\sindx-\response_\tindx) \leq 0 \; \forall \tindx,\sindx\in\{1,2,\dots,\Tindxter\}.\
				\label{eqn:AfriatFeasibilityTest}
			\end{equation}
		\item A monotonic and concave utility function that satisfies (\ref{eqn:utilitymaximization}) is given by:
			\begin{equation}
				\utility(\response) = \underset{\tindx\in \{1,2,\dots,T\}}{\operatorname{min}}\{u_\tindx+\lambda_\tindx \probe_\tindx^\p(\response-\response_\tindx)\}
				\label{eqn:estutility}
			\end{equation}
	  	\item The dataset $\mathcal{D}$ satisfies the Generalized Axiom of Revealed Preference (GARP), namely for any $\tindx \leq T$, $p_t^\p x_t \geq p_t^\p x_{t+1} \quad \forall t\leq k-1 \implies p_k^\p  x_k \leq p_k^\p  x_{1}.\qed$ 
	\end{compactenum}
\label{thm:AfriatTheorem}
\end{theorem}
The remarkable property of Afriat's Theorem is that it gives necessary and sufficient conditions for the dataset to satisfy utility maximization~\eqref{eqn:utilitymaximization}. 
The feasibility of the set of inequalities can be checked using a linear programming solver or by using Warshall's algorithm with $O(T^3)$ computations~\cite{Var06}~\cite{Var82}. 
A utility function consistent with the data can be constructed using~\eqref{eqn:estutility}. The recovered utility is not unique since any monotonic transformation of~\eqref{eqn:estutility} also satisfy Afriat's Theorem. 
\subsection{System Model}
%The general problem of utility change point detection for agents with dynamic utility function, where the utility function changes arbitrarily, is difficult to analytically characterize and therefore in this paper, we consider structured changes in utility function. 
%In this paper, we consider dynamic utility maximizing agents whose utility changes can be modelled by linear perturbations. % on a base utility function. 
In this paper, we consider an agent that maximizes a utility function that jump changes by a linear perturbation at a time that is unknown to the observer. The aim is to estimate the utility before and after the change, and the change point.  
Consider an agent that selects a response $\response$ at time $t$ to maximize the utility function given by:
\begin{equation}
	u(\response,\alpha;t) = v(\response) + \alpha^\prime \response \mathds{1}\{t\ge \tau\},
	\label{eqn:linmodel}
\end{equation}
subject to the following budget constraint $\probe_t^\prime \response \le I_t$. 
Here, $\mathds{1}\{\cdot \}$ denotes the indicator function. 
%This model is valid in the case where the utility function changes due to an external influence and there is a need to detect the change. 
The utility function, $u(\response,\alpha;t)$ consists of two components: a base utility function, $v(\response)$, and a linear perturbation, $\alpha^\prime \response$, which occurs at an unknown time $\tau$. 
The base utility function, $v(\response)$ is assumed to be monotonic and concave. 
We will restrict the components of the vector $\alpha$ to be (strictly) greater than 0, so that the utility function, $u$, conditioned on $\alpha$ is monotonic and concave. 
The objective is to derive necessary and sufficient conditions to detect the time, $\tau$ at which linear perturbation is introduced to the base utility function.  
Theorem~\ref{thm:AC:Linear} summarizes the necessary and sufficient conditions to detect the change in utility function according to the model in~\eqref{eqn:linmodel} and the proof is in Appendix~\ref{appendix:proof:AC:linear}. 
\begin{theorem}
	\label{thm:AC:Linear}
	The dataset $\mathcal{D}$ in~\eqref{eqn:dataset} is consistent with the model in~\eqref{eqn:linmodel} if we can find set of scalars $\{v_t\}_{t=1,\dots,T}$, $\{\lambda_t > 0\}_{t=1,\dots,T}$, $\{\alpha_k\}_{k=1,\dots,m}$, such that there exists a feasible solution to the following inequalities:
	\begin{align}
	 	v_t + \lambda_t \probe_t^\prime (\response_s - \response_t) 		&\ge v_s 			&  (t < \tau)   \label{eqn:IE1}\\
		v_t + \lambda_t \probe_t^\prime (\response_s - \response_t) - \alpha^\prime (\response_s -\response_t) & \ge v_s 	&  (t \ge \tau) \label{eqn:IE2}\\
		\alpha_i &\le \lambda_t \probe_t^i & (\forall i, t \ge \tau) \label{eqn:IE3},
	\end{align}
where $p_t^i$ is the $i$\textsuperscript{th} component of the probe $p_t$. 
\end{theorem}
The inequalities in~\cref{eqn:IE1,eqn:IE2,eqn:IE3} closely resemble the Afriat inequalities~\eqref{eqn:AfriatFeasibilityTest}. 
%, with $\alpha = \vec{0}$. 
The time instant $\tau$ at which the inequalities are satisfied is the time at which the linear perturbation is introduced. 
%This leads us to the Algorithm~\ref{ucdalgo} for detecting the time $\tau$, at which linear perturbation occured in the utility function. 
%Define:
%\begin{align*}
%	\mathcal{P}_t &= \left( \probe_1,\probe_2,\dots,\probe_t \right)  \\
%	\mathcal{X}_t &= \left( \response_1,\response_2,\dots,\response_t \right)
%\end{align*}

%\begin{algorithm}
%\caption{Linear Utility Change Point Detection}
%\label{ucdalgo}
%\begin{algorithmic}[1]
%	\For $t = 1$ to $T$
%  \If {$\text{GARP}(\mathcal{P}_t,\mathcal{X}_t) = 1$}
%  	\State Observations can be rationalized by a utility function
%  \Else
%  	Check feasibility of inequalities~\cref{eqn:IE1,eqn:IE2,eqn:IE3} with $\tau=t$
%	\If {Inequalities are feasible} 
%			\State Observations rationalized by the model in~\eqref{eqn:linmodel} with $\tau=t$
%		\Else
%			\State Observations cannot be rationalized
%		\EndIf
%  \EndIf
%  \State t = t + 1
%\end{algorithmic}
%\end{algorithm}
\subsection{Recovery of minimum perturbation of $\alpha$ and the base utility function}
\label{subsec:recover:alpha}
%Apart from utility change point detection, it is also important to recover the linear perturbation coefficients. 
Computing the linear perturbation coefficients in~\eqref{eqn:linmodel} gives an indication of the severity of the ground truth or the effect of marketing and advertising in social media. % introduced. 
%In this section, we propose to recover the linear perturbation coefficients by minimizing the 2-norm of $\alpha$. % to give the values of $\alpha$ that can be used to explain the data based on the model in~\eqref{eqn:linmodel}. 
%The motivation for the choice of 2-norm as an optimization criteria is explained in~\Sec~\ref{sec:recnoise}.%(Noise Analysis section). 
The solution to the following convex optimization provides the minimum value of the perturbation coefficients:
\begin{alignat}{3}
	{\text{min }}   & \norm{\alpha}_2^2 \span \label{opt:recover:alpha:no:noise} \\%\tag{$\text{OPT}_1$}\\
	 \text{s.t. }   & 
	 %v_s - v_t - \lambda_t {p}_t^\prime (x_s-x_t)  \leq 0 \; \forall s, t < \tau \label{eqn:ineqcons:1}\\
	 %&  v_s - v_t - \lambda_t {p}_t^\prime (x_s-x_t)  + \alpha^\prime (x_s - x_t) \leq 0 \; \forall s, t \geq \tau \label{eqn:ineqcons:2}\\
	 %&  \alpha_i \leq \lambda_t p_t^i \quad \forall i,  t \geq \tau \label{eqn:ineqcons:3}\\
	 \scalebox{0.95}{$v_t + \lambda_t \probe_t^\prime (\response_s - \response_t)$} & \ge \scalebox{0.95}{$v_s \phantom{\probe_t^i}   \phantom{\forall i}(t < \tau)$}   \label{eqn:ineqcons:1}\\
	 & \scalebox{0.95}{$v_t + \lambda_t \probe_t^\prime (\response_s - \response_t) - \alpha^\prime (\response_s -\response_t)$} & \ge \scalebox{0.95}{$v_s \phantom{\probe_t^i}	   \phantom{\forall i,}(t \ge \tau)$} \label{eqn:ineqcons:2}\\
	 &\scalebox{0.95}{$\phantom{v_t + \lambda_t \probe_t^\prime (\response_s - \response_t) - \alpha^\prime (\response_s -)} \alpha_i$} &\le \scalebox{0.95}{$\lambda_t \probe_t^i  (\forall i, t \ge \tau)$}  \label{eqn:ineqcons:3} \\
	 & \scalebox{0.95}{$\lambda_t > 0$}\span \nonumber\\
	 & \scalebox{0.95}{$v_1 = \beta, \quad \lambda_1 = \delta,$}\span \label{eqn:eqcons}
\end{alignat}
where, $\beta$ and $\delta$ are arbitrary constants. 

The equations~\eqref{eqn:ineqcons:1} to~\eqref{eqn:ineqcons:3} correspond to the revealed preference inequalities~\cref{eqn:IE1,eqn:IE2,eqn:IE3}. 
The normalization conditions~\cref{eqn:eqcons} are required because of the ordinality\footnotemark of the utility function. \footnotetext{Clearly any positive monotonic transformation of $u(x)$ in~\eqref{eqn:utilitymaximization} gives the same response. }
This is because for any set of feasible values of $\{v_t\}_{t=1,\dots,T}$, $\{\lambda_t\}_{t=1,\dots,T}$, $\{\alpha_k\}_{k=1,\dots,m}$ satisfying the constraints in Theorem~\ref{thm:AC:Linear} the following relation also holds
\begin{equation*}
	\beta (v_s + \delta) - \beta (v_t + \delta) - \beta \lambda_t {p}_t^\prime (x_s-x_t)  + \beta \alpha^\prime (x_s - x_t) \leq 0. 
\end{equation*}
%\subsection{Recovery of the base utility function}
%\label{subsec:recover:base}
%In this section, we recover the base utility function for the agent satisfying the model in~\eqref{eqn:linmodel}; 
Recall that the base utility function $v(x)$, is the utility function before the linear change. 
%The linear perturbation coefficients from \Sec~\ref{subsec:recover:alpha} along with the corresponding base utility function can be used to predict future behaviour of the agent. 
%Corollary~\ref{cor:recover:base} gives the recovered base utility function and the proof follows from the proof of Theorem~\ref{thm:AC:Linear}. 
\begin{corollary}
	\label{cor:recover:base}
	The recovered base utility function is
	\begin{equation}
		\hat{v}(x) = \underset{t}{\min} \{ v_t + \lambda_t \tilde{p}_t^\prime (x-x_t)\},
		\label{eqn:recover:base}
	\end{equation}
	where
	\begin{equation}
  		\tilde{p}_t^i =  \begin{cases}
  		  			p_t^i & t < \tau, \\
  		      		p_t^i - \alpha_i/\lambda_t & t \ge \tau.
  		      	  \end{cases}
  		\label{eqn:ptilde:corr}
	\end{equation}
	In~\cref{eqn:recover:base,eqn:ptilde:corr} $\{v_t\}$, $\{\lambda_t\}$, $\{\alpha_k\}$ are the solution of~\cref{opt:recover:alpha:no:noise,eqn:ineqcons:1,eqn:ineqcons:2,eqn:ineqcons:3,eqn:eqcons}. 
\end{corollary}
\subsection{Comparison with classical change detection algorithms}
\begin{table*}[t]
\normalsize
	\begin{tabular}{c|c|c|c}
		\toprule
		Method & Data model & Change model & Reference \\
		\midrule
		CUSUM &  $x_t \overset{i.i.d}{\sim} p_\theta$&  $\theta = \begin{cases} \theta_0 & t< \tau \\ \theta_1 & t \ge \tau \end{cases}$& \cite{BN93} \\
		Semi-supervised/Supervised &  \scalebox{0.9}{$\dataset = \left\{\left((p_1,I_1),U(p_1,I_1)\right),\dots,\left((p_T,I_T),U(p_T,I_T)\right)\right\}$}&  Not Applicable & \cite{BDMUV14,ZR12,BV06}\\
		Learning & \scalebox{0.9}{$\left(p_i,I_i\right) \sim P$, $U(p_i,x_i)$: Optimal response for utility $U$}&  & \cite{CSZ10}\\
		%Time Series (Linear Regression)& $y_i = a + b x_i + \varepsilon_i$&  $(a,b) = \begin{cases} (\alpha_1,\alpha_2) & t\le \tau \\ (\beta_1,\beta_2) & t > \tau \end{cases}$& \cite{Gom10} \\
		Revealed Preference & $x_t = \underset{p_t^\prime x_t \le I_t}{\argmax\;} u(x_t)$&  $u(x) = \begin{cases} v(x) & t< \tau \\ v(x) + \alpha^\prime x  & t \ge \tau \end{cases}$& This paper \\
		\bottomrule
	\end{tabular}
	\caption{Comparison of Revealed Preference with classical change detection algorithm. 
	\label{tab:compare:rp:classical}}
\end{table*}
Table~\ref{tab:compare:rp:classical} compares the revealed preference framework of our paper with classical change detection algorithms. 
The key difference is that the revealed preference framework considers a system with maximization of a utility function subject to linear constraints (the budget constraint). 
%When a parametric form of the utility function is known (for e.g.\ $v(x)$ is a Cobb-Douglas\footnotemark utility function), then the problem can be solved using a CUSUM type algorithm (see \Sec~\ref{subsec:results:change:detection:noise} for a numerical example). 
In comparison, a classical CUSUM type change detection algorithm requires knowledge of a parametrized  utility function (see \Sec~\ref{subsec:results:change:detection:noise} for a numerical example when $v(x)$ is a Cobb-Douglas\footnotemark utility function). 
The revealed preference framework for change detection makes no such assumptions. \footnotetext{The Cobb-Douglas is a widely used utility function in economics. When $m=2$, i.e.\ the dimension of the probe and response is $2$, the utility function can be expressed as $u(x) = x_1^{a} x_2^{b}$. The utility function is parameterized by $a$ and $b$. }
The revealed preference problem is related to the supervised learning literature when the parametric class of functions for empirical risk minimization (ERM) is limited to concave and monotone functions~\cite{BDMUV14,ZR12,BV06}. The change detection problem can be thought of as a multi-class learning problem with the first class being the utility function before the change and the second class being the utility function after the change.  
However, this paper provides an algorithmic approach to detect change points by deriving necessary and sufficient conditions. 

\section{Utility change point detection in noise}
\label{sec:recnoise}
\Sec~\ref{sec:ucpd} dealt with utility change point detection in the deterministic case. In this section, we consider the change point detection problem when the response of the agent is measured in noise. 
%In \Sec~\ref{subsec:classicalrevpref}, we consider the problem where the response is measure in noise under the classical revealed preference framework (static utility function). 
%In \Sec~\ref{subsec:recover:linear:perturbation}, we extend the techniques for the classical revealed preference in presence of noise, to dynamic utility maximizing agents considered in \Sec~\ref{sec:ucpd}.  
\subsection{Classical revealed preference in a noisy setting}
\label{subsec:classicalrevpref}
%\begin{figure*}[!t]
%	\normalsize
%	\hrulefill
%    %\centering
%    \begin{subfigure}{0.5\textwidth}
%        \centering
%	\includegraphics[scale=0.5]{figs/analyexplb_f.eps}
%	\caption{Upper bound on the CDF of $M$, as defined in~\eqref{eqn:def:M}. (Lower bound on the false alarm probability.)}
%	\label{fig:M:LB}
%    \end{subfigure}%
%    \begin{subfigure}{0.5\textwidth}
%        \centering
%	\includegraphics[scale=0.5]{figs/analyexpub_f.eps}
%        \caption{Upper bound on the false alarm probability.}
%	\label{fig:M:UB}
%    \end{subfigure}
%    \caption{Comparison of analytical expression with M.C. simulation}
%    \label{fig:analyexp}
%\end{figure*}
\begin{figure}[!t]
    	\centering
	\includegraphics[scale=0.5]{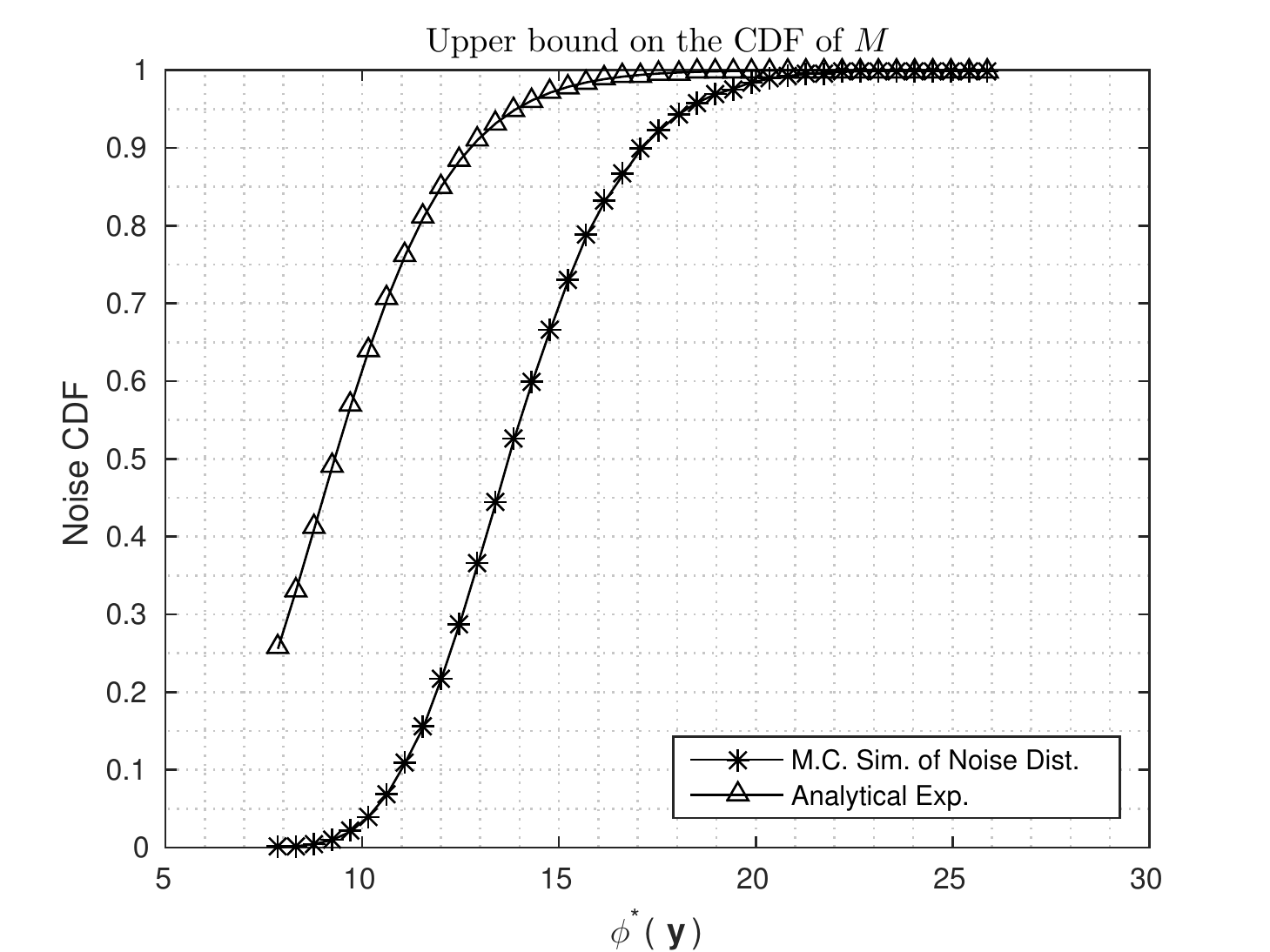}
    	\caption{\small Comparison of analytical expression with M.C.\ simulation: Upper bound on the CDF of $M$, as defined in~\eqref{eqn:def:M}. (Lower bound on the false alarm probability.)}
	\label{fig:M:LB}
\end{figure}
Afriat's theorem (Theorem~\ref{thm:AfriatTheorem}) assumes perfect observation of the probe and response. 
However, when the response of the agents are measured in noise, the failure of Afriat test could be either due to measurement noise or absence of utility maximization. 
Below, we assume the additive noise model for measurement errors given by:
%However, when the response of the agents, $x_\tindx$, are measured via noisy obervations, $y_\tindx$, according to the model:
\begin{equation}
	y_\tindx = x_\tindx + w_\tindx,
	\label{eqn:noisemodel}
\end{equation}
where $y_\tindx$ is the noisy measurement of response $x_\tindx$ and $w_\tindx \in \mathbb{R}^m$ is the independent and identically distributed (i.i.d) standard Gaussian noise\footnote{Although we consider the zero mean, unit variance Gaussian the extension to arbitrary mean $\mu$ and variance $\sigma^2$ is immediate.}. 

Given the noisy dataset 
\begin{equation}
	\obsdataset= \left\{\left(p_\tindx,y_\tindx \right): \tindx \in \left\{1,\dots,T \right\}\right\},
	\label{eqn:noisydataset}
\end{equation}
\cite{KH12} proposes the following statistical test for testing utility maximization~\eqref{eqn:utilitymaximization} in a dataset due to measurement errors. 
Let $H_0$ denote the null hypothesis that the dataset $\obsdataset$ in~\eqref{eqn:noisydataset} satisfies utility maximization. 
Similarly, let $H_1$ denote the alternative hypothesis that the dataset does not satisfy utility maximization. 
There are two possible sources of error:
\begin{align}
	\text{\bf Type-I errors:}   &\text{\hspace{1.6mm}Reject $H_0$ when $H_0$ is valid.} \nonumber\\
	\text{\bf Type-II errors:}  &\text{\hspace{1.6mm}Accept $H_0$ when $H_0$ is invalid.}
	\label{eqn:hypothesis}
\end{align}
The following statistical test can be used to detect if an agent is seeking to maximize a utility function. 
\begin{equation}
	\boxed{\phantom{\text{\hspace{0.2mm}}}\int\limits_{\Phi^*(\bf\obsresponse)}^{+\infty}f_M(\psi)\mathrm{d}\psi \overset{H_0}{\underset{H_1}{\gtrless}} \gamma}. 
	\label{eqn:Statistical_Test}
\end{equation}
In the statistical test, (\ref{eqn:Statistical_Test}): \\ (i) $\gamma$ is the ``significance level'' of the statistical test. 
\\(ii) The ``test statistic'' $\Phi^*(\bf\obsresponse)$, with ${\bf\obsresponse}=\left[\obsresponse_1,\obsresponse_2,\dots,\obsresponse_T\right]$ is the solution of the following constrained optimization problem :
\begin{equation}
\begin{array}{rl}
\min & \Phi \\
\mbox{s.t.} & u_{\sindx}-u_{\tindx}- \lambda_\tindx \probe_\tindx^\prime (\obsresponse_\sindx -\obsresponse_\tindx)-\lambda_\tindx \Phi \leq 0 \quad  \\
& \lambda_t > 0 \quad \Phi \geq 0 \quad\text{for}\quad \tindx,\sindx\in \{1,2,\dots,T\}.\label{eqn:AE}
\end{array}
\end{equation}
(iii) $f_M$ is the pdf of the random variable $M$ where
\begin{equation}
	M\triangleq\underset{\underset{\tindx \ne \sindx}{\tindx,\sindx}}{\max}\left[\probe_\tindx^\prime(\noise_\tindx - \noise_\sindx)\right]. 
	\label{eqn:def:M}
\end{equation}
The probability of false alarm or Type-I error, the probability of rejecting $H_0$, when true, is given by $\Pr\left\{M \ge \Phi^*(\bf\obsresponse)\right\}$. 

Below, we derive an analytical expression for the lower bound for the false alarm probability of the statistical test in~\eqref{eqn:Statistical_Test}. 
The motivation comes from the following fact: Given the significance level of the statistical test in~\eqref{eqn:Statistical_Test}, a Monte Carlo (M.C.) simulation is required to compute the threshold. 
However, from an analytical expression for the lower bound on false alarm probability, we can obtain an upper bound of the test statistic, denoted by $\overline{\Phi^*(\bf \obsresponse)}$. 
Hence, for any dataset $\obsdataset$ in~\eqref{eqn:noisydataset}, if the solution to the optimization problem~\eqref{eqn:AE} is such that $\Phi >  \overline{\Phi^*(\bf \obsresponse)}$, then the dataset does not satisfy utility maximization, for the desired false alarm probability. 

Theorem~\ref{thm:lowerbound} provides the lower bound on the false alarm probability and the proof is provided in Appendix~\ref{appendix:lowebound}. 
\begin{theorem}
	If the noise components have a standard Gaussian distribution, then the probability of false alarm is lower bounded by 
	\begin{equation}
		1 - \underset{t}{\prod} \left\{1 - \sqrt{\frac{2}{\pi}} \frac{\sqrt{2{\norm{p_t}}^2}\exp{(-{\Phi^*({\bf y})}^2/4{\norm{p_t}}^2)}}{{\Phi^*({\bf y})} + \sqrt{{\Phi^*({\bf y})}^2+8{\norm{p_t}}^2}}\right\}.
	  	\label{eqn:lowerbound}
	\end{equation}
	\label{thm:lowerbound}
\end{theorem}
The key idea is to bound $M$ in~\eqref{eqn:def:M} by the highest order statistic of a carefully chosen set of random variables which are \emph{negatively dependent}. Refer to Appendix~\ref{appendix:negative:dependance} for the definition of negative dependence.

%A simpler analytical expression for~\eqref{eqn:lowerbound}, can be obtained using Chernoff style bounds present in~\cite{cote2012chernoff}. %$$1 - \underset{t}{\prod} \left\{ 1 - \alpha_t \left\{\exp(-{\beta_t \Phi^*(y)}^2/4{\norm{p_t}}^2)\right\} \right\}.$$
%which is useful for analytical manipulation. 
Figure~\ref{fig:M:LB} shows the comparison of the upper bound of the cdf (and correspondingly the lower bound on the false alarm probability) and the M.C. simulation of actual density of $M$. 
As can be seen from Fig~\ref{fig:M:LB} that the upper bound of the cdf (lower bound on false alarm probability) is tight at all regimes. 
The upper bound of the test statistic, $\overline{\Phi^*(\bf \obsresponse)}$, can be obtained by setting the analytical expression in~\eqref{eqn:lowerbound} to be equal to the desired false alarm probability. 
%\subsection{Upper Bound on False Alarm Probability}
%In this section, we find an upper bound on the false alarm probability. 
%The computation of $\overline{\Phi^*(\bf \obsresponse)}$ is an offline procedure and does not depend on the dataset and can be implemented as a simple look-up table. 

\ifcompress
%A similar argument provides an upper bound on the false alarm probability and is skipped due to space constraints. 
%Figure~\ref{fig:M:UB} shows the comparison of the upper bound of the false alarm probability and the false alarm probability obtained by M.C simulation. 
\else
Appendix~\ref{subsec:upperbound:falsealarm} derives an upper bound on the false alarm probability. 
The upper bound on the false alarm probability is only asymptotically tight. 
However, in the usual region on interest (for e.g. a false alarm probability of $0.1$), the upper bound is tight. 
\fi
%\begin{figure*}[!t]
%	\normalsize
%	\hrulefill
%    %\centering
%    \begin{subfigure}{0.5\textwidth}
%        \centering
%	\includegraphics[scale=0.5]{figs/analyexplb_f.eps}
%        \caption{Upper bound on the CDF of $M$. (Lower bound on the false alarm probability.)}
%	\label{fig:M:LB}
%    \end{subfigure}%
%    \begin{subfigure}{0.5\textwidth}
%        \centering
%	\includegraphics[scale=0.5]{figs/analyexpub_f.eps}
%        \caption{Upper bound on the false alarm probability.}
%	\label{fig:M:UB}
%    \end{subfigure}
%    \caption{Comparison of analytical expression with M.C. simulation}
%    \label{fig:analyexp}
%\end{figure*}
%\subsection{Recoverability of linear perturbation coefficients for minimum false alarm probability}
\subsection{Dynamic Revealed Preference in a noisy setting}
\label{subsec:recover:linear:perturbation}
In this section, we consider the case of dynamic utility maximizing agents, satisfying the model in~\eqref{eqn:linmodel}, in the presence of noise.  
In \Sec~\ref{subsubsec:detect:changepoint:noise}, we propose a procedure to detect the unknown change point time in presence of noise. 
%Also, it is important to recover the linear perturbation coefficients and the base utility function, in addition to the unknown change point. 
In \Sec~\ref{subsubsec:recover:linear:perturbation}, similar to \Sec~\ref{subsec:classicalrevpref}, we formulate a hypothesis test to check whether the dataset satisfy the model in~\eqref{eqn:linmodel}. 
As in \Sec~\ref{subsec:classicalrevpref}, we bound the false alarm probability and obtain a criteria for recovering the linear perturbation coefficient, corresponding to minimum false alarm probability. 
Once the unknown change point time and the linear perturbation coefficients have been recovered, the base utility function can be recovered similar to that in \Sec~\ref{subsec:recover:alpha}. 
\subsubsection{Estimation of unknown change point}
\label{subsubsec:detect:changepoint:noise}
%Similar to \Sec~\ref{subsec:classicalrevpref}, we consider the case where the observed response from the dynamic utility maximizing agent are corrupted by additive noise according to~\eqref{eqn:noisemodel}. 
In the presence of noise, the inequalities in~\eqref{eqn:IE1} to~\eqref{eqn:IE3} may not be satisfied for any value of $\tau$. 
Hence, we consider the following linear programming problem, to find the minimum error or ``adjustment'' such that the inequalities in~\eqref{eqn:IE1} to~\eqref{eqn:IE3} are satisfied. 
%\begin{alignat}{3}
%	& \Phi_\tau = &&{\text{ min }} && \Phi \label{opt:min:phi:tau}\\
%        & &&\text{ s.t. } && v_s - v_t - \lambda_t \probe_t^\prime (y_s-y_t) - \Phi \leq 0,\;  \forall s, t  < \tau  \label{eqn:det:phi:1}\\
%        & && && v_s - v_t - \lambda_t \probe_t^\prime (y_s-y_t) +\alpha^\prime (y_s -y_t) -\Phi \leq 0, \;  \forall s, t  \ge \tau \IEEEeqnarraynumspace \label{eqn:det:phi:2}\\
%        & && && \alpha_i \leq \lambda_t p_t^i \quad \forall i, t \ge \tau. \nonumber
%\end{alignat}
%\begin{IEEEeqnarray}{LLL}
%	 \Phi_\tau = &{\text{ min }}  \Phi \label{opt:min:phi:tau}\\
%         \text{ s.t. } & v_s - v_t - \lambda_t \probe_t^\prime (y_s-y_t) - \Phi \leq 0,\;  \forall s, t  < \tau  \label{eqn:det:phi:1}\\
%          & v_s - v_t - \lambda_t \probe_t^\prime (y_s-y_t) +\alpha^\prime (y_s -y_t) -\Phi \leq 0, \; \nonumber\\
%	 & \hspace{15em} \forall s, t  \ge \tau \label{eqn:det:phi:2}\\
%          & \alpha_i \leq \lambda_t p_t^i \quad \forall i, t \ge \tau \quad \Phi \ge 0 \nonumber
%\end{IEEEeqnarray}
\begin{alignat}{3}
	{\Phi_\tau}=   &   \text{ min }  \Phi \label{opt:min:phi:tau}\\
	 \text{s.t. }   & 
	 \scalebox{0.95}{$v_s - v_t - \lambda_t \probe_t^\prime (y_s - y_t) - \Phi $} & \le \scalebox{0.95}{$0\quad \phantom{\forall i,}(t < \tau)$} \nonumber \\
	 & \scalebox{0.95}{$v_s - v_t - \lambda_t \probe_t^\prime (y_s - y_t) + \alpha^\prime (y_s -y_t) - \Phi$} & \le \scalebox{0.95}{$0 \quad \phantom{\forall i,}(t \ge \tau)$} \nonumber \\
	 &\hspace{12em}\scalebox{0.95}{$\alpha_i - \lambda_t \probe_t^i $} &\le \scalebox{0.95}{$0  \quad (\forall i, t \ge \tau)$}  \nonumber \\
	 & \hspace{6em} \Phi \ge 0, \quad \lambda_t > 0 \span \nonumber
\end{alignat}
%\begin{equation}
%\makebox[0pt]{\begin{minipage}{\linewidth}
%\begin{IEEEeqnarray}{LLL}
%	\Phi_\tau =  & {\text{ min }}  \Phi \nonumber\\
%        &  \hspace{-2.5em} \text{s.t. } \; v_s - v_t - \lambda_t \probe_t^\prime (y_s-y_t) - \Phi \leq 0,\;  \forall s, t  < \tau   \nonumber \\
%	&  \hspace{-2.5em}v_s - v_t - \lambda_t \probe_t^\prime (y_s-y_t) +\alpha^\prime (y_s -y_t) -\Phi \leq 0, \;  \forall s, t  \ge \tau \IEEEeqnarraynumspace \nonumber\\
%        &  \hspace{-2.5em} \alpha_i \leq \lambda_t p_t^i \quad \forall i, t \ge \tau \quad \Phi \ge 0. \nonumber
%\end{IEEEeqnarray}
%\end{minipage}}
%\label{opt:min:phi:tau} \tag{$\text{OPT}_2$}
%\end{equation}
The solution of the linear program~\eqref{opt:min:phi:tau} depends on the choice of the change point variable $\tau$. 
When the data is measured without noise, the equations are satisfied with zero error at the correct change point. 
The estimated change point, $\hat{\tau}$, corresponds to time point with minimum adjustment. 
\begin{equation}
	\hat{\tau} = \underset{1\le \tau \le T}{\argmin\;} \Phi_\tau
	\label{eqn:opt:change:point}
\end{equation}
The intuition for~\eqref{eqn:opt:change:point} is if ${\tau}$ is the true change point, then the perturbation $\Phi$ needs to compensate only for the noise. 

\subsubsection{Recovering the linear perturbation coefficients for minimum false alarm probability}
\label{subsubsec:recover:linear:perturbation}
As in~\eqref{eqn:hypothesis}, define the null hypothesis $H_0$, that the dataset satisfies utility maximization under the model in~\eqref{eqn:linmodel}, and the alternative hypothesis $H_1$ that the dataset does not satisfy utility maximization under the model in~\eqref{eqn:linmodel}. 
Type-I errors and Type-II errors are defined, similarly, as in \Sec~\ref{subsec:classicalrevpref}. 

Consider the following hypothesis test: 
\begin{equation}
	\boxed{\phantom{\text{\hspace{0.2mm}}}\int\limits_{\Phi^*(\bf\obsresponse)}^{+\infty}f_M(\psi)\mathrm{d}\psi \overset{H_0}{\underset{H_1}{\gtrless}} \gamma}.
	\label{eqn:newstattest}
	%\int_{\Phi^*({\bf y})}^{+\infty} f_M(\beta) d\beta \gtrless_{H_1}^{H_0} \gamma. 
\end{equation}
In~\eqref{eqn:newstattest}: 
\begin{enumerate}
  \item $\gamma$ is the significance level of the test. 
  \item The test statistic ``$\Phi^*({\bf y})$'', is the solution of the following constrained optimization problem with $\tau = \hat{\tau}\text{ (from \Sec~\ref{subsubsec:detect:changepoint:noise})}$.
	  %\begin{alignat}{2}
	  %        &{\text{min }} && \Phi \label{opt:min:phi} \tag{$\text{OPT}_3$}\\
	  %        & \text{s.t. } && v_s - v_t - \lambda_t \probe_t^\prime (y_s-y_t) - \lambda_t \Phi \leq 0, \label{eqn:noisyAC1}\\
	  %        & && v_s - v_t - \lambda_t \probe_t^\prime (y_s-y_t) +\alpha^\prime (y_s -y_t) -\lambda_t \Phi \leq 0,\; \forall s, t  \ge \tau \label{eqn:noisyAC2}\\
	  %        & && \alpha_i \leq \lambda_t p_t^i \quad \forall i, t \quad \Phi \ge 0. \nonumber
	  %\end{alignat}
	  %\begin{IEEEeqnarray}{LLL}
	  %        {\text{min }} & \Phi \label{opt:min:phi} \\ %\tag{$\text{OPT}_3$}\\
	  %         \text{s.t. } & v_s - v_t - \lambda_t \probe_t^\prime (y_s-y_t) - \lambda_t \Phi \leq 0, \label{eqn:noisyAC1}\\
	  %         & v_s - v_t - \lambda_t \probe_t^\prime (y_s-y_t) +\alpha^\prime (y_s -y_t) -\lambda_t \Phi \leq 0,\nonumber \\
	  %         & \hspace{15em} \forall s, t  \ge \tau \label{eqn:noisyAC2}\\
	  %         & \alpha_i \leq \lambda_t p_t^i \quad \forall i, t \ge \tau \quad \Phi \ge 0. \nonumber
	  %\end{IEEEeqnarray}
	\begin{alignat}{3}
		\text{ min }  & \Phi \label{opt:min:phi} \span \\
		 \text{s.t. }   & 
		 \scalebox{0.9}{$v_s - v_t - \lambda_t \probe_t^\prime (y_s - y_t) - \lambda_t\Phi $} & \le \scalebox{0.95}{$0\; \phantom{\forall i,}(t < \tau)$} \nonumber \\
		 \span \scalebox{0.95}{$v_s - v_t - \lambda_t \probe_t^\prime (y_s - y_t) + \alpha^\prime (y_s -y_t) - \lambda_t\Phi$} & \le \scalebox{0.95}{$0 \; \phantom{\forall i,}(t \ge \tau)$} \nonumber \\
		 &\hspace{12em}\scalebox{0.95}{$\alpha_i - \lambda_t \probe_t^i $} &\le \scalebox{0.95}{$0  \; (\forall i, t \ge \tau)$}  \nonumber \\
		 & \hspace{6em} \Phi \ge 0, \quad \lambda_t > 0 \span \nonumber
	\end{alignat}
	  The optimization problem above~\eqref{opt:min:phi} is similar to the optimization problem~\eqref{opt:min:phi:tau} but is introduced for simplicity of analysis of the error probability in the subsequent sections. 
      \item $f_M$ is the pdf of the random variable, $M$, where
	\begin{equation}
		M \triangleq M_1 + M_2, 
	\end{equation}
	where $M_1$ and $M_2$ are defined as:
	\begin{align}
		M_1 &\triangleq \underset{\underset{s \ne t}{s,t}}{\max} \left[p_t^\prime (w_t - w_s)\right], \\
		M_2 &\triangleq \underset{\underset{s \ne t, t \ge \tau}{s,t}}{\max} \left[\alpha^\prime {(w_t - w_s)}/\lambda_t\right], 
	\end{align}
\end{enumerate}
where $\alpha$ and $\lambda$ are the solution of~\eqref{opt:min:phi}.
%The statistical test that we present here do not fall into classical literature on non-parametric change point detection  
%The set of ``noisy'' Afriat inequalities in~\eqref{opt:min:phi} can be written as follows:
The set of inequalities in~\eqref{opt:min:phi} can be re-written using~\eqref{eqn:noisemodel} as:
%The set of ``noisy'' Afriat inequalities, (31) and (32), can be written as follows:
\begin{equation*}
\begin{aligned}
	&(v_s - v_t)/\lambda_t - {p}_t^\prime (x_s  - x_t ) \le  p_t^\prime {\left(w_t - w_s \right)}  & \quad (t < \tau), \\
	&(v_s - v_t)/\lambda_t - {p}_t^\prime (x_s  - x_t ) + \alpha^\prime  (x_s - x_t )/\lambda_t \nonumber \span \\ 
	\span \quad \le  p_t^\prime {\left(w_t - w_s \right)} - \alpha^\prime {\left(w_t - w_s \right)}/\lambda_t \; & \quad (t \ge \tau). 
\end{aligned}
\end{equation*}
Therefore for the dataset satisfying the model in~\eqref{eqn:linmodel} it should be the case that the test statistic ``$\Phi^*({\bf y}) \le M$''. 

The pdf of the random variable $M$ can be computed using Monte Carlo simulation. As in \Sec~\ref{subsec:classicalrevpref}, we bound the probability of false alarm and obtain a criteria to recover the linear perturbation coefficients corresponding to minimum false alarm probability. Theorem~\ref{thm:alphacriteria} provides the criteria and the proof is provided in Appendix~\ref{appendix:proof:alphacriteria}. 
\begin{theorem}
	\label{thm:alphacriteria}
	Assume that the noise components are iid zero mean unit variance Gaussian. Suppose $\tau={O}(\log{1/\varepsilon})$ and $T-\tau={O}(\log{1/\varepsilon})$, where $\varepsilon > 0$. Then the optimization criterion to recover $\alpha$ with minimum probability of Type-I error is to minimize $\|\alpha\|_2$ (i.e., the Euclidean norm).  
\end{theorem}
The recovery of the linear perturbation coefficients and the base utility function are similar to that in \Sec~\ref{subsec:recover:alpha}.  
\section{Dimensionality reduction in Revealed Preference} 
\label{sec:dimensionreduction}
Classical revealed preference deals with the case $m < T$ (recall $m$ is the dimension of the probe vector and $T$ is the number of observations). 
Below, we consider the ``big data'' domain: $m \gg T$.  
Checking whether a dataset, $\dataset$, satisfies utility maximization~\eqref{eqn:utilitymaximization} can be done by verifying whether GARP (statement 4 of Theorem~\ref{thm:AfriatTheorem}) is satisfied. 
For $m \gg T$, the computational cost for checking GARP is dominated by the number of computations required to evaluate the inner product in GARP, given by $mT^2$. The computational cost for computing the inner product can be reduced by \emph{embedding} the $m$-dimensional probe and observation vector into a lower dimensional subspace, of dimension~$k$, and checking GARP on the lower dimensional subspace. 
We use Johnson-Lindenstrauss Lemma (JL Lemma) to achieve this. 
%Informally, the JL Lemma states that for a set of $n$ vectors in a high dimensional space, there exists a mapping to a lower dimensional space (the embedding space) such that the length and pairwise distance of the mapped vectors is within a tolerable range of the original vectors. The key advantage of JL Lemma is that the dimension of the embedding space can be surprisingly small and does not depend on the dimension of the vectors.  
%Formally, the JL Lemma is stated below in Lemma~\ref{lem:JL}. 
\begin{lemma}[Johnson-Lindenstrauss {(JL)}~\cite{JL84}]
	Suppose $x_1,x_2,\dots,x_n \in \reals^m$ are $n$ arbitrary vectors. For any $\varepsilon \in (0,\frac{1}{2})$, there exists a mapping $f:\reals^m\rightarrow\reals^k$, $k=O(\log n/\varepsilon^2)$, such that the following conditions are satisfied:
	\begin{alignat}{2}
		&(1-\varepsilon) {\norm{x_i}}^2 &\le {\norm{f(x_i)}}^2  \le  (1+\varepsilon) {{\norm{x_i}}}^2 \quad \forall i \label{eqn:normpreserveJL} \\
		&(1-\varepsilon) {\norm{x_i-x_{j}}}^2 &\le {\norm{f(x_i)-f(x_{j})}}^2  \nonumber \\
		&\quad &\le  (1+\varepsilon) {{\norm{x_i-x_{j}}}}^2  \quad \forall i,j. \label{eqn:distancepreserveJL}
	\end{alignat}
	\label{lem:JL}
\end{lemma}\vspace{-1em}
%Here, $m$ is the dimension of the high dimensional space, $k$ is the dimension of the embedding space and $\varepsilon$ is the tolerable error. 
%In addition, if $f$ is a random linear map then the condition~\eqref{eqn:distancepreserveJL} holds with probability atleast $1-\frac{2}{n}$. 
%For a fixed $\varepsilon$, the number of operations required to project a vector of dimension $m$, to a vector of dimension $t$, is given by $O(t\log n)$. 
%Since there are $n$ vectors, the total number of operat
To implement JL efficiently, one possible method of~\cite{DA03} is summarized in Theorem~\ref{thm:JLimp}. 
%, where the projection matrix is constructed with either $+1$ or $-1$ with equal probability. 
This method utilizes a linear map for $f$ and hence can be represented by a projection matrix $R$. 
The key idea in~\cite{DA03} is to construct the projection matrix $R$ with elements $+1$ or $-1$ so that the computing the projection involves no multiplications (only additions). 
\begin{theorem}[\cite{DA03}]
	\label{thm:JLimp}
	Let $ A = [x_1,x_2,\dots,x_n]^\prime$ denote the $n \times m$ data matrix. 
	Given $\varepsilon,\beta > 0$, let $R$ be a $m \times k$ random binary matrix, with independent and equiprobable elements $+1$ and $-1$, where %as follows: 
	%\begin{equation*}
	%	R(i,j) = \begin{cases}
	%		+1 & \text{with probability $1/2$} \\
	%		-1 & \text{with probability $1/2$}, 
	%		\end{cases}
	%\end{equation*}
	%with 
	\begin{equation}
		k > \frac{4+2\beta}{\varepsilon^2/2 - \varepsilon^3/3} \log n. 
		\label{eqn:achlioptasJL}
	\end{equation}
	The projected data matrix, $B$ of dimension $n \times k$, is given by $B = \frac{1}{\sqrt{k}} A R$. 
	Then with probability at least $1-\delta$, where $\delta=\frac{1}{n^\beta}$, the inequalities~\cref{eqn:normpreserveJL,eqn:distancepreserveJL} holds, where $f:\reals^m\rightarrow\reals^k$ maps the $i$\textsuperscript{th} row of $A$ to the $i$\textsuperscript{th} row of $B$. \qed
\end{theorem}
%As before, $\varepsilon$ denotes the tolerable error in~\cref{eqn:normpreserveJL,eqn:distancepreserveJL} and $k$ denotes the dimension of the embedded space ($k_0$ denotes the minimum dimension). 
The inequalities in~\cref{eqn:normpreserveJL,eqn:distancepreserveJL} hold in a probabilistic sense (with probability $1-\delta$), with the parameter $\beta$ controlling the corresponding probability.  

Checking the GARP conditions (statement 4 of Theorem~\ref{thm:AfriatTheorem}) depends only on the relative value of the inner product between the probe and response vectors. 
Hence, we can scale both the probe and response vector such that their norms are less that one. 
In this case, as a consequence of preservation of the norms of the vector, the Johnson-Lindenstrauss embedding also preserves the inner product. 
\begin{corollary}
	Let $x_i, x_j \in \reals^m$ and $\norm{x_i} \le 1, \norm{x_j} \le 1$ be such that~\eqref{eqn:distancepreserveJL} is satisfied with probability at least $1-\delta$. Then,
	\begin{equation}
	\Pr\left(\left(x_i^\prime x_j - f(x_i)^\prime f(x_j)\right) \ge \varepsilon \right) \le \delta. \nonumber
	\end{equation}
	\label{corr:JLinnerprod}
\end{corollary}
\ifcompress \vspace{-1em} The proof is available, for example, in~\cite{VEM05}.  \else \begin{proof} Proof included in appendix.\end{proof} \fi
The JL embedding of the vectors preserves the inner product to within a $\varepsilon$ fraction of the original value. 

Therefore, to check for utility maximization behaviour, we first project the high dimensional probe and response vector to a lower dimension using JL (using Theorem~\ref{thm:JLimp}). 
The inner products in the lower dimensional space is then used for checking the GARP condition for detecting utility maximization giving $\frac{m}{k}$ savings in computation. 
%Note that the dimension of the embedding subspace, $k$, is independent of the dimension of the original space, $m$, and hence higher savings can be obtained when the dimension of original space is high. 
\section{Numerical Results}
\label{sec:results}
The aim of this section is three fold. 
First, we illustrate the change point detection algorithm in \Sec~\ref{sec:recnoise} and show how the revealed preference framework considered in this paper is fundamentally different from classical change detection algorithms. 
Second, we show that the theory developed in \Sec~\ref{sec:ucpd} and \Sec~\ref{sec:recnoise}, for utility change point detection, can successfully predict the change in ground truths through online search behaviour.  
Also, the recovered utility functions satisfy the single crossing condition indicating strategic substitution behaviour\footnote{The substitution behaviour in economics is the idea that consumers, constrained by a budget will substitute more expensive items with less costly alternatives. } in online search.  
Third, we show user behaviour in YouTube satisfies utility maximization. 
To reduce the computational cost associated with checking the utility maximization behaviour, we use dimensionality reduction techniques discussed in \Sec~\ref{sec:dimensionreduction}. 
In addition, in \Sec~\ref{sec:predYouTube}, we provide an algorithm to predict total traffic in YouTube. 
%To the best of our knowledge, such non-parametric prediction is new. 
\subsection{Detection of unknown change point in the presence of noise}
\label{subsec:results:change:detection:noise}
In this section, we present simulation results on change point detection in the presence of noise. 
For the simulation study, assume that the probe and response vector is of dimension $2$, i.e.\ $m=2$. 
Assume that the system follows the model given by~\eqref{eqn:simulation:cobb:douglas}. 
The base utility function $v(x)$ is a Cobb-Douglas utility function with parameter $a_1$ and $a_2$. 
\begin{equation}
	\begin{aligned} 
		v(x) &= x_1^{a_1} x_2^{a_2} \\
		u(x) &= \begin{cases}
			v(x) & t < \tau \\
			v(x) + \alpha^\prime x & t \ge \tau
			\end{cases}
	\end{aligned}
	\label{eqn:simulation:cobb:douglas}
\end{equation}
%For $t < \tau$, the optimization problem~\eqref{eqn:utilitymaximization} can be solved in closed form and is given by $(x_t^1, x_{t}^{2}) = (\frac{a_1}{a_1 + a_2} \frac{I_t}{p_{t}^{1}},\frac{a_2}{a_1 + a_2} \frac{I_t}{p_{t}^{2}})$. 
%For $t \ge \tau$, there is no closed form solution for the optimization problem~\eqref{eqn:utilitymaximization}. 
Let the response be measured in noise as defined in~\eqref{eqn:noisemodel}. 

\Fig~\ref{fig:phi:tau} shows the simulation results for $\Phi_\tau$ in~\eqref{opt:min:phi:tau} as a function of $\tau$. The estimated change point ($\hat{\tau}$) is the point at which $\Phi_\tau$ attains minimum.  
\Fig~\ref{fig:roc:rp:cusum} compares the ROC plot for the revealed preference framework and the CUSUM algorithm for change detection. The details of the CUSUM algorithm for change detection are provided in Appendix~\ref{appendix:cusum:utility:change:point}. The CUSUM algorithm is used as a reference for comparing the performance of the revealed preference framework presented in this paper. 
The CUSUM algorithm in Appendix~\ref{appendix:cusum:utility:change:point} makes two critical assumptions: \begin{inparaenum}[(i)] \item Knowledge of the utility function before change \item Knowledge of the linear perturbation coefficients, and hence the utility function after the change\end{inparaenum}. 
The only unknown is the change point at which the utility changed. However, if the linear perturbation coefficients are \emph{also} unknown, then the CUSUM algorithm in Appendix~\ref{appendix:cusum:utility:change:point} can be modified to search over $\reals^m_+$ and select the parameter with the highest likelihood. The critical assumption is the knowledge of the utility function before the change point. One heuristic solution is to estimate the utility function using some initial data, assuming no change point, utilizing the Afriat's Theorem and then applying the CUSUM algorithm. Such a procedure is clearly suboptimal. In comparison, the revealed preference procedure in \Sec~\ref{subsec:recover:linear:perturbation} makes no assumption about the base utility function or the linear perturbation coefficients. As can be gleaned from \Fig~\ref{fig:roc:rp:cusum}, the performance of the revealed preference algorithm is comparable to the CUSUM algorithm, given the non-parametric assumptions. 
\begin{figure}[h]
	\centering
	\begin{subfigure}{0.45\textwidth}
		\includegraphics[width=\textwidth]{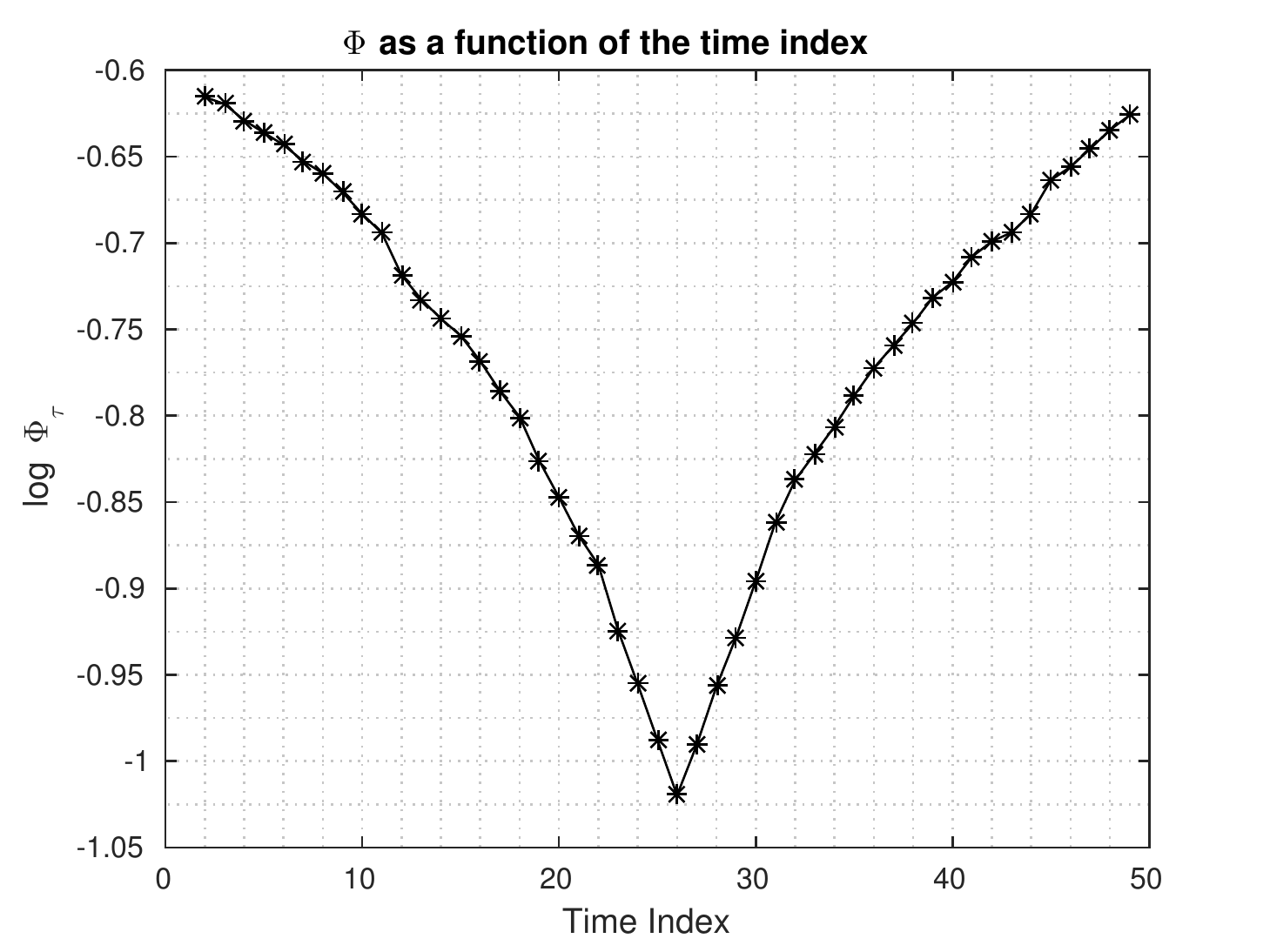}
		\caption{\small $\Phi_\tau$ (as defined in~\eqref{opt:min:phi:tau}) as a function of the change point $\tau$. }
		\label{fig:phi:tau}
	\end{subfigure}
    \begin{subfigure}{0.45\textwidth}
        \centering
  	\includegraphics[width=\textwidth,keepaspectratio]{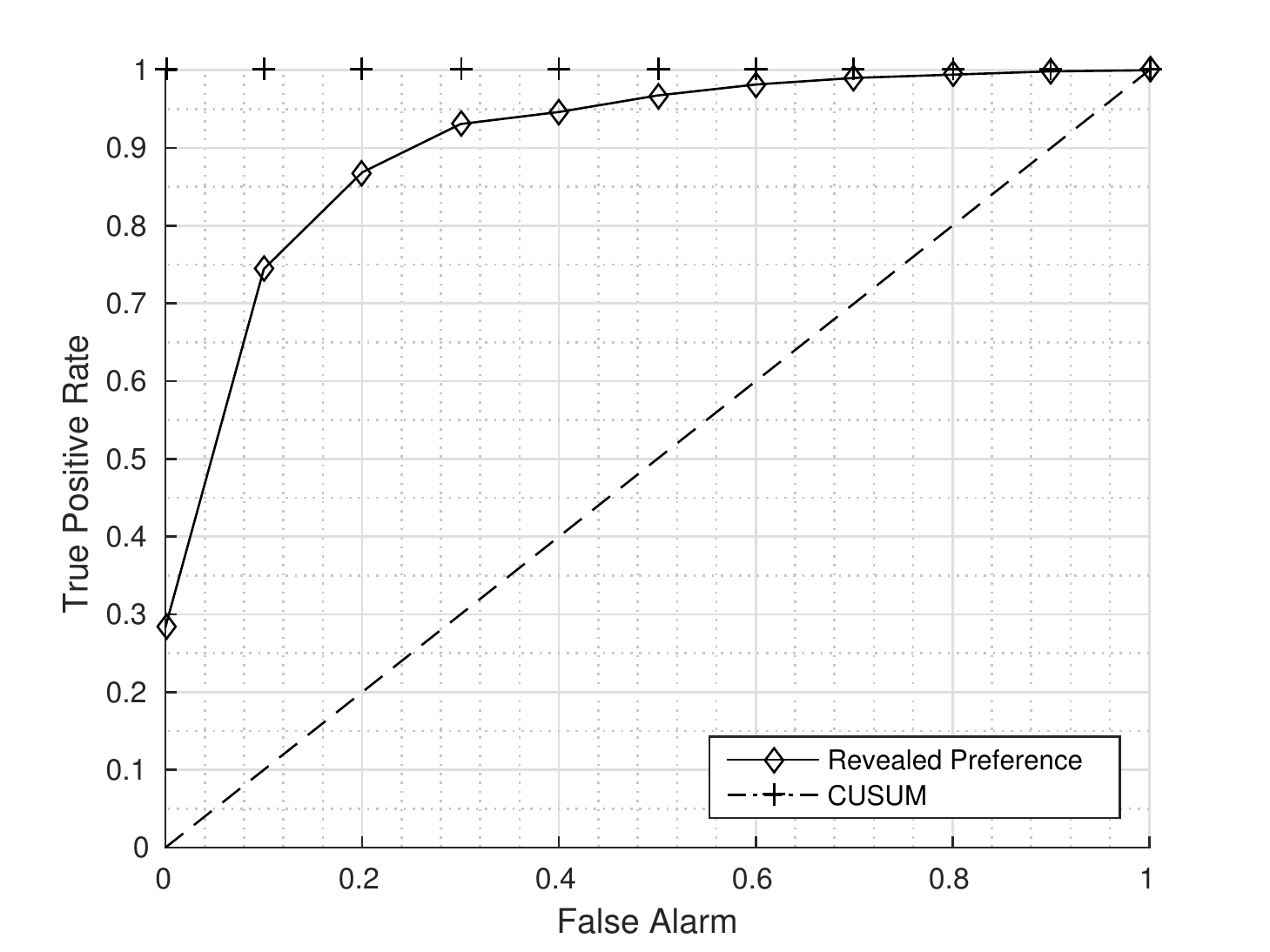}
	\caption{\small ROC plot. }
  	\label{fig:roc:rp:cusum}
	%\label{fig:cusum}
    \end{subfigure}
    \caption{\small Estimation of the utility change point using the revealed preference framework (\Fig~\ref{fig:phi:tau}).  \Fig~\ref{fig:roc:rp:cusum} compares the ROC plots of the revealed preference framework with the CUSUM algorithm. The CUSUM algorithm in Appendix~\ref{appendix:cusum:utility:change:point} assumes knowledge of the utility function before and after the change point. However, the revealed preference framework considered in this paper assumes no parametric knowledge of the utility function. The plots were generated with $1000$ independent simulations. The parameters of Cobb-Douglas utility $v(x)$ equal to $(a_1, a_2)=(0.6,0.4)$ and $\alpha=(1,1)$ with the change point set as $26$. The budget is set to $5$. The noise variance is $0.50$.}
\label{fig:comparison:afriat:noise:cusum}
\end{figure}
\subsection{Yahoo! Buzz Game}
\label{subsec:yahooresults}
%In this section, we present an example of a real dataset of a utility maximizing agent with a dynamic utility function. 
In this section, we present an example of a real dataset of online search process. 
%The objective to detect change points or unknown time at which ground truth has changed. 
The objective is to investigate the utility maximization of the online search process and to detect change points at which the utility has changed. 
The change points give useful information on when the ground truths have changed. 

The dataset that we use in our study is the Yahoo! Buzz Game Transactions from the Webscope datasets\footnote{Yahoo! Webscope dataset: A2 - Yahoo! Buzz Game Transactions with Buzz Scores, version 1.0 \url{http://research.yahoo.com/Academic_Relations}} available from Yahoo! Labs. 
In 2005, Yahoo!~along with O'Reilly Media started a fantasy market where the trending technologies at that point where pitted against each other. 
For example, in the browser market there were ``Internet Explorer'', ``Firefox'', ``Opera'', ``Mozilla'', ``Camino'', ``Konqueror'', and ``Safari''. 
The players in the game have access to the ``buzz'', which is the online search index, measured by the number of people searching on the Yahoo! search engine for the technology. 
The objective of the game is to use the buzz and trade stocks accordingly.  
The interested reader is referred to~\cite{MDFHKPD05} for an overview of the Buzz game. 
%The dataset above is a classical example of prediction markets. 
An empirical study of the dataset~\cite{CPK08} reveal that most of the traders in the Buzz game follow utility maximization behaviour. 
Hence, the dataset falls within the revealed preference framework, if we consider the buzz as the probe and the ``trading price\footnotemark'' as the response to the utility maximizing behaviour.  \footnotetext{The trading price is indicative of the value of the stock. } 

%Therefore, in this section, we investigate the utility function underlying the online search process. 
We consider a subset of the dataset containing only the ``WIRELESS'' market which contained two main competing technologies: ``WIFI'' and ``WIMAX''. 
Figure~\ref{fig:buzz} shows the buzz and the ``trading price'' of the technologies starting from April 1 to April 29. 
The buzz is published by Yahoo! at the start of each day and the ``trading price'' was computed as the average of the trading price of the stock for each day. 
\begin{figure}[h]
  \centering
  %\raggedleft  
  \includegraphics[width=0.5\textwidth,keepaspectratio]{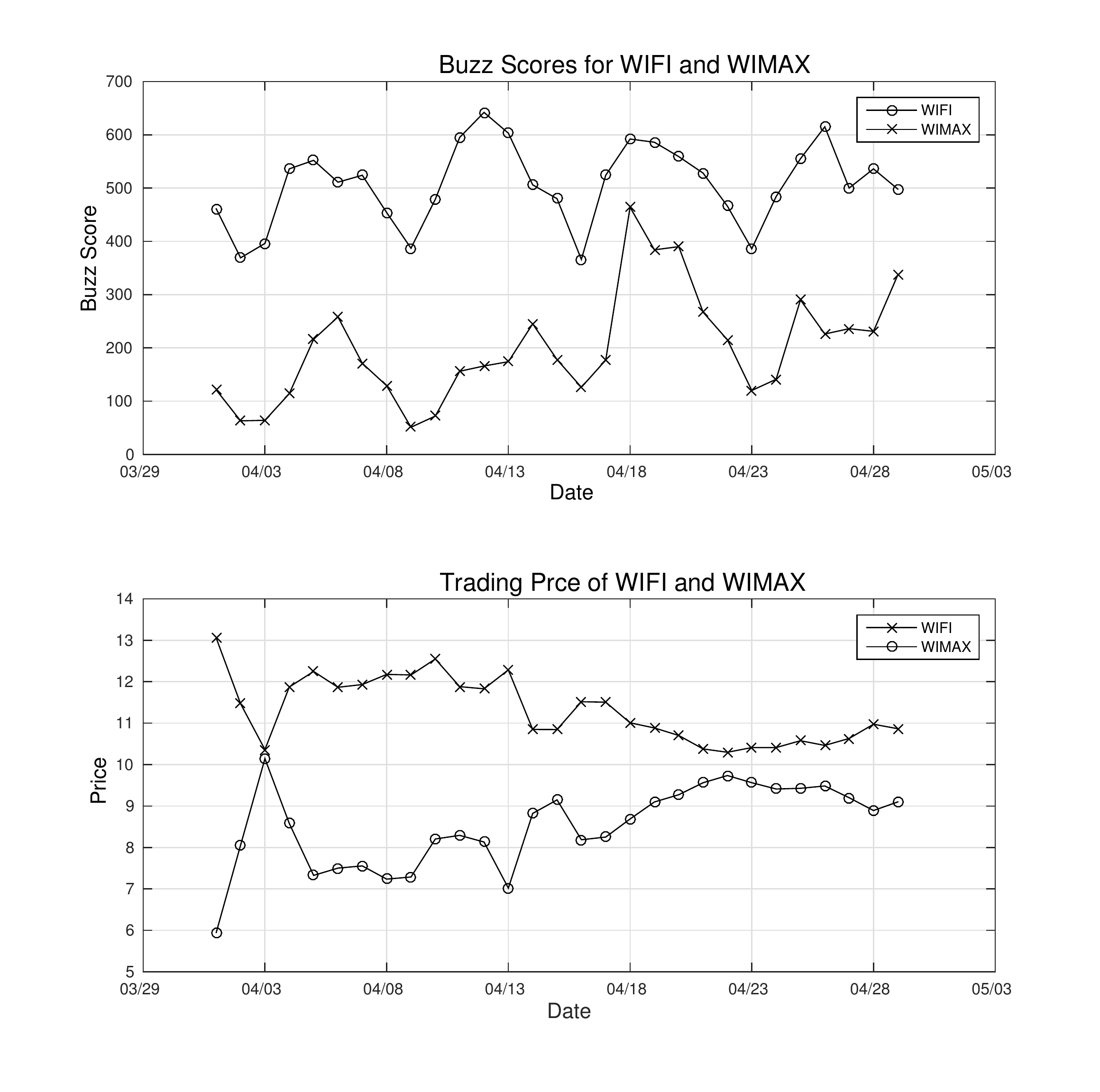}
  \caption{\small Buzz scores and trading price for WIFI and WIMAX in the WIRELESS market from April 1 to April 29. The change point was estimated as April 18. This corresponds to a new WIFI product announcement. The change can also be observed due to the sudden peak of interest in WIFI around April 18.}
  \label{fig:buzz}
\end{figure}

Chose the probe and response vector for this dataset as follows
\begin{equation*}
	\begin{aligned}
		\probe_\tindx &= \left[\text{Buzz(WIFI)}\quad \text{Buzz(WIMAX)}\right] \\
		\response_\tindx &= \left[\text{Trading price(WIFI)}\quad \text{Trading price(WIMAX)}\right].
	\end{aligned}
\end{equation*}
%The buzz index are normalized to a value of $100$ for all the competing technologies and hence the product $p_t^\prime x_t$ determines the total worth 	
Checking the GARP condition or the Afriat inequalities~\eqref{eqn:AfriatFeasibilityTest}, we find that the dataset does not satisfy utility maximization for the entire duration from April 1 to April 29. 
However, we find that the dataset satisfies utility maximization from April 1 to April 17. 
Using the inequalities~\cref{eqn:IE1,eqn:IE2,eqn:IE3}, that we derived in Sec~\ref{sec:ucpd}, for the model in~\eqref{eqn:linmodel}, we see that utility has changed with change point, $\tau$, set to April 18. 
%Hence the utility has changed from a perfectly competitive market and we suspect this was caused due to exogenous process. 
This correspond to a change in the ground truth which affected the utility of the agents. 
Indeed, we find that the change point corresponds to Intel's announcement of WIMAX chip\footnote{\url{http://www.dailywireless.org/2005/04/17/intel-shipping-wimax-silicon/}}. 

%Furthermore, by recovering linear coefficients which correspond to minimum perturbation we get $\alpha = [0 \; 5.9]$, which corresponds to a positive change in the WIMAX utility. 
Also, by minimizing the 2-norm of the linear perturbation, which we derived in \Sec~\ref{sec:recnoise}, and using the optimization problem, postulated in \Sec~\ref{subsec:recover:alpha}, we find that the recovered linear coefficients which correspond to minimum perturbation is $\alpha = [0 \; 5.9]$. 
This is inline with what we expect, a positive change in the WIMAX utility, due to the change in ground truth. 
Furthermore, the recovered utility function, $v(x)$, is shown in \Fig~\ref{fig:recovered_utility} and the indifference curve (contour plot) of the base utility is shown in \Fig~\ref{fig:indifference_recovered_utility}. 
\begin{figure*}[t]
	\normalsize
    \centering
    \begin{subfigure}{0.5\textwidth}
        \centering
	\includegraphics[scale=0.45,valign=t]{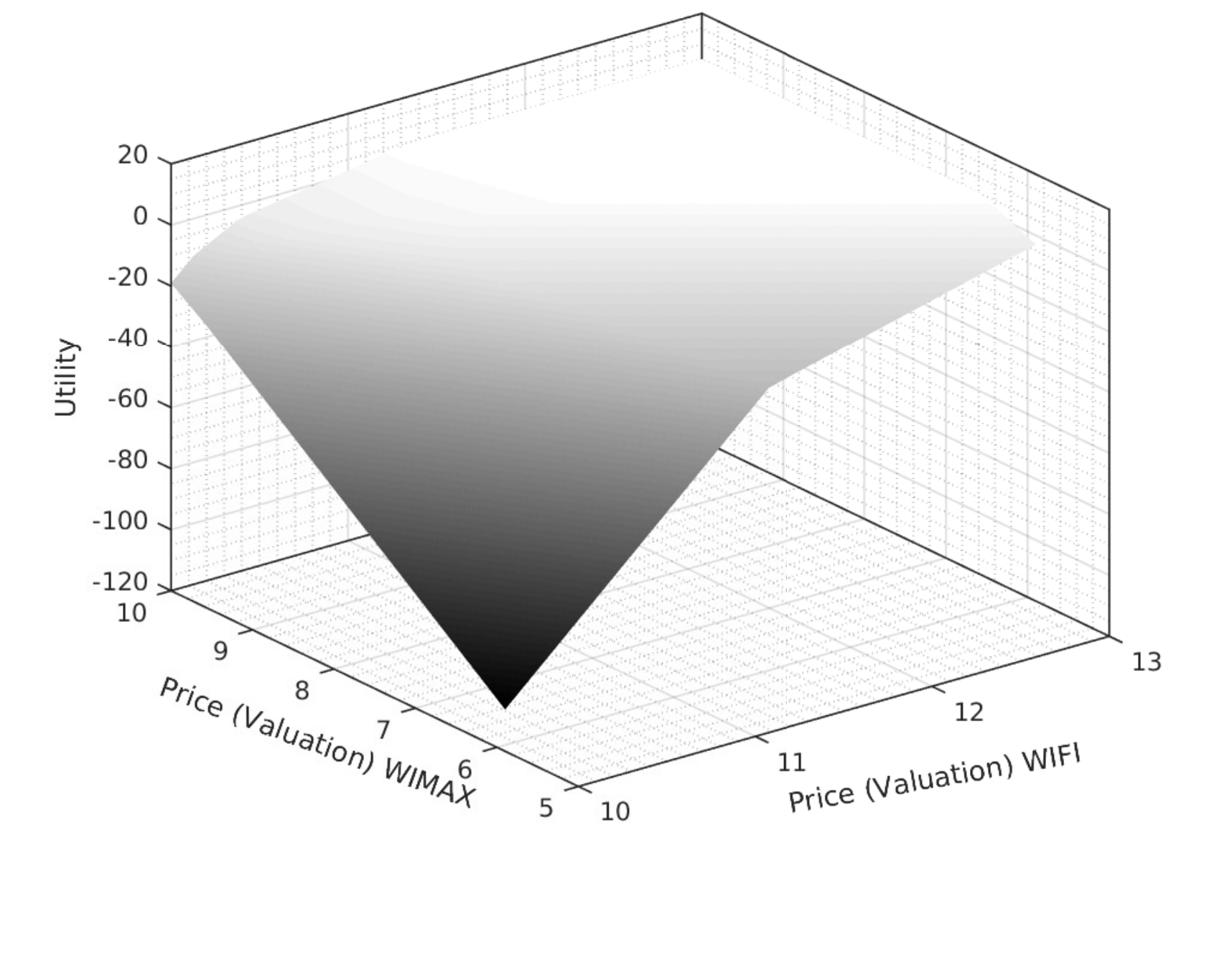}
	\caption{}
	\label{fig:recovered_utility}
    \end{subfigure}%
    \begin{subfigure}{0.5\textwidth}
        \centering
	\includegraphics[scale=0.5,valign=t]{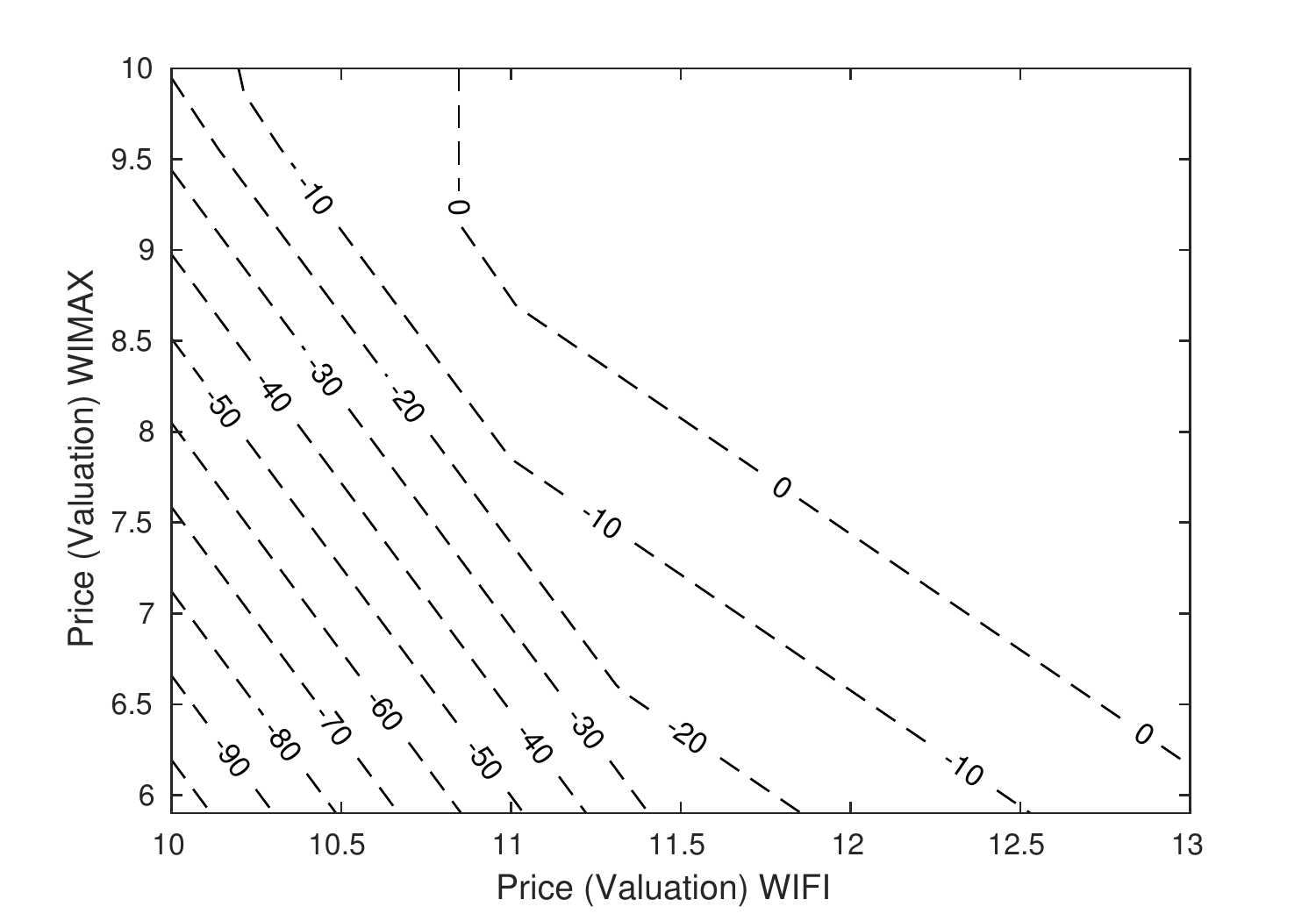}
        \caption{}
	\label{fig:indifference_recovered_utility}
    \end{subfigure}
    \caption{\small \Fig~\ref{fig:recovered_utility} shows the recovered utility function $v(x)$ using~\cref{eqn:recover:base}. Indifference curve of the recovered utility function, is shown in \Fig~\ref{fig:indifference_recovered_utility}.  The indifference curve suggests the substitution behaviour in online search. }
    \label{fig:baseutilityfn}
\end{figure*}
The recovered base utility function in \Fig~\ref{fig:recovered_utility} satisfy the single crossing condition\footnotemark indicating strategic substitute behaviour in online search. \footnotetext{Utility function, $U(x_1,x_2)$, satisfy the single crossing condition if $\forall \; x_1^\prime > x_1, \; x_2^\prime > x_2$, we have $U(x_1^\prime,x_2)-U(x_1,x_2) \ge 0 \implies U(x_1^\prime,x_2^\prime) - U(x_1,x_2^\prime) \ge 0$. The single crossing condition is an ordinal condition and therefore compatible with Afriat's test.}
The substitute behaviour in online search can also be noticed from the indifference curve in \Fig~\ref{fig:indifference_recovered_utility}. %  is typical of a utility function for perfect substitutes. 
This is due to the fact that WIFI and WIMAX were competing technologies for the problem of providing wireless local area network. 
%The change point detection that we presented in \Sec~\ref{sec:ucpd} could be used to detect the onset of epidemic disease such as influenza 
%
%Similar to the example dataset that we presented here, the change point detection framework that we presented in \Sec~\ref{sec:ucpd} could be used to detect the onset of epidemic diseases such as influenza or change in user opinion due to major political decision~\cite{CM09,ZYF11,SSGA10,Doo10,WB09}. 
\subsection{Youstatanalyzer Database}
We now analyze the utility maximizing behaviour of users engaged in the popular online platform for videos, YouTube. 
YouTube is an example of a content-aware utility maximization, where the utility depends on the quality of the video present at any point. 
We measure the quality of the video using two measurable metrics: the number of subscribers and the number of views. 
%Due to the large amount of user-generated video content, the data from YouTube is high dimensional. 
%One of the key characteristics of YouTube is the sharing of user-generated video content on top of a social network. 
%It is well known that the social network effects of diffusion, contagion, homophily and conformity play cruical role in the success\footnote{In this paper, we measure success by the number of views.} of a video~\cite{SOT12}. 
%Also, it is well known that number of views of a video on YouTube is dependent on the social network effects of diffusion, contagion, homophily and conformity~\cite{SOT12}. 
%Initially the effect of contagion causes the number of views to increase with the number of subscribers~\cite{KS85}\footnote{In economics literature, \cite{KS85} shows that the utility function of the $n+1$ adopter of a contagious phenomenon is strictly increasing in the number of adopters, $n$.}. 
%As the video ages, the effect of conformity and homophily causes the number of views to increase with the number of subscribers~\cite{SOT12}. 
%This contagion effect in turn influences the utility of all users in the network. 
%These social network effects in turn influences the utility of all users in the network. 
%In addition, homophily of the YouTube users ensures that all users have similar utility functions. 
%Hence, in this paper, we study the aggregated utility function of all users. 
The YouTube database that we use for our study is the Youstatanalyzer database\footnote{http://www.congas-project.eu/sites/www.congas-project.eu/files/Dataset/youstatanalyzer1000k.json.gz}. 
The Youstatanalyzer database is built using the Youstatanalyzer tool~\cite{ZMP13}, which collects statistics of YouTube videos using web scrapping technology. 
The database is particularly suited for our study of dimensionality reduction in revealed preference, since the entire database contains statistics of 1000K videos. 

From the database, we aggregated the statistics of all popular videos existing from start of 08 July, 2012 to end of 07 Sept, 2013, having at least $2$ subscribers. 
The entire duration was divided into 15 time periods, corresponding to each month of the duration, giving us a total of $15$ observations ($T = 15$). 
The entire duration need to be split into sub-time periods since the statistics of all the videos are not sampled each day. 
For the revealed preference analysis, each of the dimension of the probe and the response is associated with a unique video ID. 
The probe for the revealed preference analysis is the number of subscribers and the response is the number of views during the time period. 
The objective is to investigate utility maximizing behaviour between the number of people subscribed to a particular video and the number of views that the video received. 
Formally, the probe and the response is chosen as follows:
\begin{align*}
	\probe_\tindx 		&= \left[1/\text{\#Subscriber}(\text{Video}{1}),\dots,1/\text{\#Subscriber}(\text{Video}N)\right], \\ 
	\response_\tindx 	&= \left[\text{\#Views}(\text{Video}1),\dots,\text{\#Views}(\text{Video}N)\right]. 
\end{align*}
The motivation for this definition is that as the number of subscribers to a video increases, the number of views also increase~\cite{CFMZZL14}. 
%A recent study~\cite{CFMZZL14} on YouTube found that there is a high correlation (almost $0.998$) between the number of views and the number of subscribers. 
The inner product of the probe and the response vector gives the sum of the ``view focus'' of all videos~\cite{BSW12}. 
Also, a recent study shows that $90\%$ of the YouTube views are due to $20\%$ of the videos uploaded~\cite{DDHLWR11}. 
Hence, if we restrict attention to popular videos, the view focus tend to remain constant during a time period which correspond to the linear constraint in the revealed preference setting. 
%Hence, the inner product $p_t\prime x_t$ closely approximate the total amount of YouTube traffic and should be less than the total traffic than can be handled by the network, $I_t$.  

The number of videos satisfying the above requirements is $7605$, and therefore, the dimension of the probe and response, $m = 7605$. 
%The number of computations required for compute the GARP condition ($\mathcal{O}(T^3)$) is smaller than the number of inner product computations required for creating the GARP conditions ($T^2 m$). 
The number of inner product computations required for checking the GARP conditions is given by $m  T^2$.  
Hence, we apply Johnson-Lindenstrauss lemma to the data using the ``database friendly'' transform that we presented in Sec~\ref{sec:dimensionreduction}. 
We choose $\varepsilon = 0.1$, so that the inner product are within 90\% of the accuracy. 
Also, we choose $\beta=0.65$, such that the above condition on the inner products hold with probability at least $0.9$. 
Substituting the values of $T$, $\varepsilon$ and $\beta$ in~\eqref{eqn:achlioptasJL} we find that the dimension of the embedded subspace is $k=3800$. 
From the simulations, we see that the GARP is satisfied with probability $0.9$, which is inline with what we expect.  
For this example, the number of inner product computations required to compute GARP condition in the lower dimensional subspace is given by $k T^2$, which is less than the number of computations required to compute the GARP in the original space by a factor of $2$. 
%Note that even though the computational advantage is modest in this case, the dimension of the embedding is not dependent on the dimension of the probe and response and higher computational advantage can be obtained when the dimension of the probe and the response are higher.  
Note from Lemma~\ref{lem:JL} and Theorem~\ref{thm:JLimp} that the dimension of the embedding is independent of the dimension of the probe and the response and hence higher computational savings can be obtained when the dimension of the probe and the response are higher. 

The utility function obtained in the lower dimension is useful for visualization, and gives a sparse representation of the original utility function. 
Below, we use the utility function obtained in the lower dimension to predict total traffic to YouTube when the number of subscribers to a particular video changes or when popular YouTube users (people with large number of subscribers) upload new videos. 
The above approach of predicting total number of views complements the findings in~\cite{CFMZZL14}, where the authors claim that individual views to YouTube video or channel cannot be predicted. 
This predicted total traffic serves as an useful benchmark for allocating server resources without compromising the user experience. 
\subsubsection*{\bf Predicting total traffic in YouTube}
\label{sec:predYouTube}
%In this section, we provide an algorithm to predict the total traffic to YouTube, 
Based on the utility function recovered in the lower dimension, one can predict the total traffic in YouTube.  
%The constructive proof of Afriat's theorem~\cite{Afr67} enables us to recover a piecewise linear function that rationalizes the dataset. 
% The utility function in~\eqref{eqn:estutility} is a piecewise linear function that rationalizes the dataset. 
As before, let $p_t \in \reals^m$ and $x_t \in \reals^m$ be the number of subscribers and the number of views at time $t$, respectively. 
Let $\tilde{p_t} \in \reals^k$ and $\tilde{x_t} \in \reals^k$ be the lower dimensional embedding of $p_t$ and $x_t$, respectively. 
The total traffic, or the number of views, at time $t$, is given by $T_t = \sum_{i=1}^{m}x_t^i$. 

We can estimate the total traffic using the lower dimension utility function. 
Given $p$, the number of subscribers at time $t+1$, let $\tilde{p}$ denote the lower dimensional embedding of $p$. 
The corresponding value of $\tilde{x}$ that rationalizes the data can be obtained by solving the following optimization problem: 
\begin{equation}
	\tilde{x} =  \underset{\{\tilde{p}^\prime x \le I\}}{\argmax} \underset{i \in \left\{1,2,\dots,t\right\}}{\operatorname{min}}\{u_i+\lambda_i \tilde{\probe_i}^\p(\response-\tilde{\response_i})\},
	\label{eqn:predictutility}
\end{equation}
The values of $\left\{u_i,\lambda_i\right\}$ are such that the GARP condition~\eqref{eqn:AfriatFeasibilityTest} is satisfied for the dataset $\left\{(\tilde{p_i},\tilde{x_i})_{i=1,2,\dots,t}\right\}$. 
The budget $I$, is assumed to be known or estimated independently of the above optimization problem. 
The estimated traffic due to the probe $\tilde{p}$ is given by $\tilde{T}=\sum_{i=1}^k \tilde{x}^i$. 

%\emph{\it \bf Simulation Results:} We now test the prediction algorithm above using synthetic data generated using Cobb-Douglas utility function. The budget is set to $1$. In order to illustrate the dimensionality reduction obtained through the application of Johnson-Lindenstrauss, we consider $1000$ random probe vectors and the corresponding response vectors (i.e.\ $T=1000$). The dimension of the embedded space depends on the tolerable error $\varepsilon$. We choose $\varepsilon = 0.25$ and the resulting dimension of the embedded space, from~\eqref{eqn:achlioptasJL}, is $150$. The utility function in the lower dimension was estimated using~\eqref{eqn:estutility}. For random probe vectors, the response was found by solving~\eqref{eqn:predictutility}. From simulation, we find that the total traffic estimated from the low dimensional space is close to $1\pm \varepsilon$ of the true value of the total traffic. An accurate analysis of the error due to the above estimation procedure is outside the scope of this paper. 
%%The utility function can be used to predict future traffic to YouTube when the number of subscribers to a particular video changes or when popular YouTube users (people with large number of subscribers) upload new videos. The prediction could be done at a lower dimensional subspace, since the total number of views given by $\vec{1}^\prime x_t$ is preserved under JL transform, to within a tolerable error.   
\section{Conclusion}
\label{sec:conclusion}
This paper derived a revealed preference based approach for change point detection in the utility function.  
The main result (Theorem~\ref{thm:AC:Linear}) provided necessary and sufficient conditions for the existence of the change point at which the utility function jump changes by a linear perturbation. 
In addition, in the presence of noise, we provided a procedure for detecting the unknown change point and a hypothesis test for testing the dataset for dynamic utility maximization. 
The results were applied on the Yahoo! Tech Buzz dataset and the estimated change point corresponds to the change in ground truth. 

%We, then, considered the problem of detection of utility maximization behaviour in the big data domain. 
%As an example, we considered the utility maximization behaviour in YouTube. 
%In order to reduce the computational cost, we proposed dimensionality reduction through the Johnson-Lindenstrauss transform. 
%Apart from reducing the computational cost, the utility function obtained in the lower dimension can be used to estimate the total traffic to YouTube. 
%
The application of results developed in this paper provided novel insights into the utility maximizing behaviour of agents in online social media. 
Extension of the current work could involve analytically characterizing the GARP failure rate due to dimensionality reduction, or considering multiple change points, or considering higher order perturbation functions. 
\section*{Acknowledgment}
%\begin{acks}
	The authors would like to thank William Hoiles\footnote{{ECE Department, University of British Columbia}} for valuable discussions and the datasets. 
%\end{acks}
\appendix
%\section*{APPENDIX}
\setcounter{section}{1}
\subsection{Proof of Theorem~\ref{thm:AC:Linear}}
\label{appendix:proof:AC:linear}
Necessary Condition: Assume that the data has been generated by the model in~\eqref{eqn:linmodel}. 
An optimal interior point solution to the problem must satisfy:
\begin{align} 
	\bigtriangledown_{x_t^i} v(x_t) &= \lambda_t p_t^i  &(t < \tau) \label{eqn:firstorder1} \\
	\bigtriangledown_{x_t^i} v(x_t) + \alpha_i 	&= \lambda_t p_t^i  &(t \ge \tau) \label{eqn:firstorder2}
\end{align}
At time $t$, the concavity of the utility function implies:
\begin{equation}
	u(x_t,\alpha,t) + \bigtriangledown_{x_t}u(x_t,\alpha,t)^\prime (x_s -x_t) \ge  u(x_s,\alpha,t) \quad \forall s. 
	\label{eqn:concavity}
\end{equation}
Substituting the first order conditions~\eqref{eqn:firstorder1}, ~\eqref{eqn:firstorder2} into~\eqref{eqn:concavity}, yields
\begin{align}
	v(x_t) + \lambda_t p_t^\prime (x_s - x_t) 		&\ge v(x_s) 			\;  (t < \tau)   \label{eqn:AC1}\\
	v(x_t) + \lambda_t p_t^\prime (x_s - x_t) - \alpha^\prime (x_s -x_t) & \ge v(x_s) 	\;  (t \ge \tau) \label{eqn:AC2}
\end{align}
Denoting $v(x_t)=v_t$ yields the set of inequalities~\eqref{eqn:IE1},~\eqref{eqn:IE2}. 
\eqref{eqn:IE3} holds since the utility function $v(x)$ is monotonic increasing. 

Sufficient Condition: We first construct a piecewise linear utility function $\mathcal{V}(x)$ from the lower envelope of the $T$ overestimates, to approximate the function $v(x)$ defined in~\eqref{eqn:linmodel},
\begin{equation}
	\mathcal{V}(x) = \underset{t}{\text{min}} \{ v_t + \lambda_t \tilde{p}_t^\prime (x-x_t)\}, \label{eqn:V:approx}
\end{equation}
where each coordinate of $\tilde{p_t}$ is defined as,
\begin{equation}
  \tilde{p}_t^i =  \begin{cases}
    			p_t^i & t < \tau \\
			p_t^i - \alpha_i/\lambda_t & t \ge \tau
		  \end{cases}
  \label{eqn:ptilde}
\end{equation}
To verify that the construction in~\eqref{eqn:V:approx} is indeed correct, consider an arbitrary response, $\hat{x}$, such that: $p_t^\prime \hat{x} \le p_t^\prime {x}_t$\footnotemark. \footnotetext{In microeconomic theory, $x_t$ is said to be ``revealed preferred'' to $\hat{x}$. Since $x_t$ was chosen as response for the probe $p_t$, the utility at $x_t$ should be higher than the utility at $\hat{x}$. } 
We need to show $\mathcal{V}(\hat{x})+\alpha^\prime \hat{x} \le \mathcal{V}(x_t)+\alpha^\prime x_t$. 

First, we show that $\mathcal{V}(x_t) = v_t \forall t$ as follows. 
\begin{equation*}
  \mathcal{V}(x_t)  = v_m + \lambda_m \tilde{p}_m^\prime (x_t - x_m),
\end{equation*}
for some $m$. 
\ifcompress
If, $m \ge \tau$,
\begin{align*}
  \mathcal{V}(x_t) 	&= v_m + \lambda_m \tilde{p}_m^\prime (x_t - x_m) \\
  		 	&= v_m + \lambda_m p_m^\prime (x_t - x_m) - \alpha^\prime (x_t - x_m)\\
  		 	& \le v_t + \lambda_t p_t^\prime (x_t - x_t) \\
  			&= v_t
\end{align*}
If the inequality is true, then it would violate~\eqref{eqn:AC2}. 
Using similar technique, we obtain, if $m < \tau$, $\mathcal{V}(x_t)=v_t$. 
Hence, $\mathcal{V}(x_t) = v_t$. 
\else
If $m < \tau$,
\begin{align*}
  \mathcal{V}(x_t) 	&= v_m + \lambda_m \tilde{p}_m^\prime (x_t - x_m) \\
  		 	&= v_m + \lambda_m p_m^\prime (x_t - x_m) \\
  		 	& \le v_t + \lambda_t p_t^\prime (x_t - x_t) \\
  			&= v_t
\end{align*}
If the inequality is true, then it would violate~\eqref{eqn:AC1}. 
Similarly, for $m \ge \tau$,
\begin{align*}
  \mathcal{V}(x_t) 	&= v_m + \lambda_m \tilde{p}_m^\prime (x_t - x_m) \\
  		 	&= v_m + \lambda_m p_m^\prime (x_t - x_m) - \alpha^\prime (x_t - x_m)\\
  		 	& \le v_t + \lambda_t p_t^\prime (x_t - x_t) \\
  			&= v_t
\end{align*}
If the inequality is true, then it would violate~\eqref{eqn:AC2}. 
Hence, $\mathcal{V}(x_t) = v_t$. 
\fi

\ifcompress
Next, we show $\mathcal{V}(\hat{x})+\alpha^\prime \hat{x} \le \mathcal{V}(x_t)+\alpha^\prime x_t$. 
If, $t \ge \tau$,
\begin{align*}
  \mathcal{V}(\hat{x}) + \alpha^\prime \hat{x} 	&\le  v_t + \lambda_t \tilde{p}_t^\prime (\hat{x}-x_t) + \alpha^\prime \hat{x}\\
  						&= v_t + \lambda_t {p}_t^\prime (\hat{x}-x_t)  -\alpha^\prime (\hat{x} - x_t) + \alpha^\prime \hat{x} \\
  						&= v_t + \lambda_t {p}_t^\prime (\hat{x}-x_t)  + \alpha^\prime x_t \\
						&\le v_t + \alpha^\prime x_t \\
						&= \mathcal{V}(x_t) + \alpha^\prime x_t 
\end{align*}
The inequality holds, similarly, for the case $t < \tau$. 
\else
Next, we show that the constructed utility function results in a higher utility for $x_t$ which is revealed preferred to $\hat{x}$. 
For $t < \tau$, 
\begin{align*}
  \mathcal{V}(\hat{x})  	&\le  v_t + \lambda_t \tilde{p}_t^\prime (\hat{x}-x_t) \\
  				&= v_t + \lambda_t {p}_t^\prime (\hat{x}-x_t)  \\
				&\le v_t \\
				&= \mathcal{V}(x_t) 
\end{align*}
For, $t \ge \tau$,
\begin{align*}
  \mathcal{V}(\hat{x}) + \alpha^\prime \hat{x} 	&\le  v_t + \lambda_t \tilde{p}_t^\prime (\hat{x}-x_t) + \alpha^\prime \hat{x}\\
  						&= v_t + \lambda_t {p}_t^\prime (\hat{x}-x_t)  -\alpha^\prime (\hat{x} - x_t) + \alpha^\prime \hat{x} \\
  						&= v_t + \lambda_t {p}_t^\prime (\hat{x}-x_t)  + \alpha^\prime x_t \\
						&\le v_t + \alpha^\prime x_t \\
						&= \mathcal{V}(x_t) + \alpha^\prime x_t 
\end{align*}
\fi
Therefore, we can construct a utility function that rationalizes the data based on the model in~\eqref{eqn:linmodel}. \qed
\subsection{Negative Dependence of Random Variables}
\label{appendix:negative:dependance}
\begin{definition}[\cite{GKV12}]
  	\label{def:negative:random}
	Random variables $X_1,\dots,X_n,n\ge2,$ are said to be \emph{negatively dependent}, if for any numbers $x_1,\dots,x_n$, we have $$\Pr\left\{\cap_{k=1}^n\left\{X_k \le x_k\right\}\right\} \le \prod_{k=1}^n \Pr\left\{X_k \le x_k\right\},$$ and $$\Pr\left\{\cap_{k=1}^n\left\{X_k > x_k\right\}\right\} \le \prod_{k=1}^n \Pr\left\{X_k > x_k\right\}.\qed$$ 
\end{definition}
The interesting characteristic of negative dependence is that it allows us to bound the joint distribution of the random variables with their marginals. 

The variable $M$ in~\eqref{eqn:def:M} is the highest order statistic of the set of random variables $\mathcal{M}$ defined as:
\begin{equation}
  \mathcal{M}\triangleq \left\{\left(\probe_\tindx^\prime (\noise_\tindx - \noise_\sindx)\right) :\sindx,\tindx = \left\{1,2,\dots,\Tindxter \right\}, \sindx \ne \tindx \right\}. 
  \label{eqn:defM}
\end{equation}
Define, $\xi \subset \mathcal{M}$ as 
\begin{equation}
	\xi = \left\{\probe_1^\prime(\noise_1-\noise_2), \probe_2^\prime(\noise_2-\noise_3)\dots,\probe_T^\prime(\noise_T-\noise_1)\right\}. 
	\label{eqn:defxi}
\end{equation}
%We next show that the set of random variables in $\xi$ is negative dependent
%
%

%
%
%. 
%%%%%%%%%%%%%%TODO
%From~\eqref{eqn:Statistical_Test}, we see that the probability of false alarm depends on the probability density of random variable $M$. 
%Computing an exact distribution of the random variable $M$, is difficult due to the dependence of the random variables in $\mathcal{M}$. 
%Hence, in the following we derive upper bound for the cumulative density function (cdf) of the random variable $M$, and correspondingly obtain lower bound on the false alarm probability. 
%\subsection{Lower Bound on False Alarm Probability}
%In this section, we find a upper bound on the cdf of $M$ and correspondingly obtain a lower bound on the false alarm probability. 
%
%
% and proof of negative dependence of random variables in $\xi$ (Lemma~\ref{lem:negativedep}). 
%
%Since $\xi \subset \mathcal{M}$, we have, 
%\begin{equation*}
%  	\Pr\left(\max \{\mathcal{M}\} \le x\right) \le \Pr\left(\max \{\xi\} \le x \right) = \Pr\left(\cap_{k}\left\{\xi_k \le x\right\}\right) . 
%\end{equation*}
%Hence we can bound the joint distribution above with the marginals. 
%Since each of the marginal distributions are Gaussian, we can use standard tail bounds for Gaussian such as the one in~\cite{AS72}. 
%%%%%%%%%%% TODO
\begin{lemma}
	If the noise components are i.i.d zero mean unit variance Gaussian distribution, then the set of random variables in $\xi$ are negatively dependent. 
	\label{lem:negativedep}
\end{lemma}
\begin{proof}
Each of the random variables in the set $\xi$, (defined in~\eqref{eqn:defxi}), is Gaussian and hence, to show negative dependence of random variables in $\xi$, it is sufficient to show that these variables are negatively correlated~\cite{GKV12,JP83}. 
%The $i$\textsuperscript{th} element in $\xi$, $p_i^\prime(w_i-w_{i+1})$, is correlated with two types:
Any element in $\xi$, $p_i^\prime(w_i-w_{i+1})$, is correlated with either:
\begin{compactenum}
	\item Element of the form $p_{i+1}^\prime(w_{i+1}-w_{i+2})$:
		\begin{equation*}
			\E\left\{\left(p_i^\prime(w_i-w_{i+1})\right)\left(p_{i+1}^\prime(w_{i+1}-w_{i+2})\right)\right\} %\\
			%&= \E\left\{\left(p_i^\prime(w_i-w_{i+1})\right)\left((w_{i+1}-w_{i+2})^\prime p_{i+1}\right)\right\} \\
			%&= p_i^\prime \E \left\{(w_i-w_{i+1})(w_{i+1}-w_{i+2})^\prime \right\} p_{i+1} \\
			%\shortintertext{Expanding the terms,}
			= -p_i^\prime p_{i+1} < 0. 
		\end{equation*}
	\item Element of the form $p_k^\prime(w_k-w_{k+1}),\; k\notin\{i,i+1\}$:
		\begin{equation*}
			\E\left\{\left(p_i^\prime(w_i-w_{i+1})\right)\left(p_{k}^\prime(w_{k}-w_{k+1})\right)\right\} %\\
			%&= \E\left\{\left(p_i^\prime(w_i-w_{i+1})\right)\left((w_{k}-w_{k+1})^\prime p_{k}\right)\right\} \\
			%&= p_i^\prime \E\left\{\left(w_i-w_{i+1}\right)\left(w_{k}-w_{k+1}\right)^\prime\right\} p_{k} \\
			%\shortintertext{Expanding the terms,}
			%&\E \left\{p_i w_i p_{k} w_{k} - p_i w_i p_{k}w_{k+1} -  p_i w_{i+1} p_{k}w_{k+1} + p_i w_{i+1} p_{k} w_{k+1} \right\} \\
			= 0. 
		\end{equation*}
\end{compactenum}
Hence the random variables in $\xi$~\eqref{eqn:defxi} are negatively correlated and hence, negatively dependent, as defined in Def.~\ref{def:negative:random}.
\end{proof}
\subsection{Proof of Theorem~\ref{thm:lowerbound}}
\label{appendix:lowebound}
%From the definition of $M$,
For any subset of the random variables, $\xi$, $$\xi \subset \mathcal{M} = \left\{\left({p}_t^\prime (w_t-w_s)\right) :s,t = \left\{1,2,\dots,T\right\}, s\ne t\right\}$$
\begin{align*}
	&\Pr \left\{M \le x \right\} 	%&\le \Pr \left\{\hat{M} \le x \right\} \\
	\le \Pr \left\{\underset{i}{\max\;} \xi_i  \le x\right\}  
	= \Pr \left\{\xi_1  \le x, \dots, \xi_T \le x\right\} \\ 
	\shortintertext{Choosing the set $\xi$ to be set defined in~\eqref{eqn:defxi}. Also, from Lemma~\ref{lem:negativedep} the random variables in $\xi$ are negatively dependent, as defined in Def.~\ref{def:negative:random}. Hence,}
	&\le \underset{i}{\prod} \Pr \left\{\xi_i \le x \right\} \\
	\shortintertext{Each of the term in $\xi$, $\left({p}_t^\prime (w_t-w_{t+1})\right)$ is distributed as $\mathcal{N}(0,2\norm{p_t}^2)$. 
	Using standard lower bound for the tail of the Gaussian distribution, we have}
	&\le \underset{t}{\prod} \left\{1 - \sqrt{\frac{2}{\pi}} \frac{\sqrt{2{\norm{p_t}}^2}}{x + \sqrt{x^2+8{\norm{p_t}}^2}}\exp(-x^2/4{\norm{p_t}}^2)\right\}  %\tag*{\qedhere} %\qed
\end{align*}
The false alarm probability is given by $1-\Pr \left\{M \le \Phi^*({\bf y}) \right\}$. 
Substituting the upper bound for $\Pr \left\{M \le  \Phi^*({\bf y})\right\}$, we get a lower bound for the false alarm probability.\qed 
%\subsection{Proof of Theorem~\ref{thm:upperbound}}
\ifcompress
% Do not include this section
\fi
\subsection{Proof of Theorem~\ref{thm:alphacriteria}}
\label{appendix:proof:alphacriteria}
The proof of Theorem~\ref{thm:alphacriteria} relies on two lemmas: Lemma~\ref{lem:M1M2:alwayspositive} and Lemma~\ref{lem:M1M2closeexpectation} which are stated below. 
Lemma~\ref{lem:M1M2:alwayspositive} states that for ``sufficient'' number of observations the random variables are ``almost'' positive. 
%\begin{lemma}
%	\label{lem:M1M2:alwayspositive}
%	If the noise components have a standard Gaussian distribution, and if $T = {O}(\log{1/\varepsilon})$, where $\varepsilon > 0$, then $$\Pr \left\{M_1 \le 0\right\} < \varepsilon.$$
%	Similarly, if $T - \tau = {O}(\log{1/\varepsilon})$ then, $$\Pr \left\{M_2 \le 0\right\} < \varepsilon.\qed$$ 
%\end{lemma}
\begin{lemma}
	\label{lem:M1M2:alwayspositive}
	Assume that the noise components are i.i.d zero mean unit variance Gaussian random variables. For $\varepsilon > 0$, $T = {O}(\log{1/\varepsilon})$ and $T - \tau = {O}(\log{1/\varepsilon})$, we have:
	\begin{align*}
		\Pr \left\{M_1 \le 0\right\} &< \varepsilon, \\
		\Pr \left\{M_2 \le 0\right\} &< \varepsilon \qed
	\end{align*}
\end{lemma}
Define, auxiliary random variables, $\hat{M_1}$ and $\hat{M_2}$, which corresponds to the truncated distributions of $M_1$ and $M_2$ as shown below: 
\begin{equation}
	f_{\hat{M_i}}(x) = f_{M_i}(x) \mathds{1}\left\{x \ge 0\right\} + \Pr\left(M_i < 0 \right) \delta(x) \; \text{for } i = 1,2, 
	\label{eqn:m1:m2:hat}
\end{equation}
where, $\delta(x)$ is the delta function. 
Then, Lemma~\ref{lem:M1M2closeexpectation} states that the expectation of the auxiliary random variables $\hat{M_i};i=1,2$, are close to the expectation of the original random variables, $M_i;i=1,2$. 
%\begin{lemma}
%	\label{lem:M1M2closeexpectation}
%	If $T={O}(\log{1/\varepsilon})$, the expectation of $\hat{M_1}$ is close to  the expectation of $M_1$ as follows $$|\E{\hat{M_1} - \E{M_1}}| < 2\varepsilon.$$
%	Similarly, if $T-\tau={O}(\log{1/\varepsilon})$, the expectation of $\hat{M_2}$ is close to  the expectation of $M_2$ as follows $$|\E{\hat{M_2} - \E{M_2}}| < 2\varepsilon.$$ \qed
%\end{lemma}
\begin{lemma}
	\label{lem:M1M2closeexpectation}
	Assume that the noise components are i.i.d zero mean unit variance Gaussian random variables. For $\varepsilon > 0$, $T = {O}(\log{1/\varepsilon})$ and $T - \tau = {O}(\log{1/\varepsilon})$, we have:
	\begin{align*}
		|\E{\hat{M_1} - \E{M_1}}| &< 2\varepsilon,\\
		|\E{\hat{M_2} - \E{M_2}}| &< 2\varepsilon . \qed
	\end{align*}
\end{lemma}
The proof of Lemma~\ref{lem:M1M2:alwayspositive} and Lemma~\ref{lem:M1M2closeexpectation} are provided in Appendix~\ref{appendix:proof:lem:M1M2:alwayspositive} and Appendix~\ref{appendix:proof:lem:M1M2closeexpectation}, respectively. 
\begin{proof}[Proof (Theorem~\ref{thm:alphacriteria})]
	For $ \Phi^*({\bf y})\ > 0$, the probability of Type-I error is given by $\Pr\left\{M \ge \Phi^*({\bf y})\right\}$. 
\begin{align*}
	&\Pr\left\{M \ge \Phi^*({\bf y})\right\} 	= \Pr\left\{M_1 + M_2 \ge \Phi^*({\bf y})\right\} \\
	\shortintertext{If $\tau = {O}(1/\varepsilon)$, by Lemma~\ref{lem:M1M2:alwayspositive} and Lemma~\ref{lem:M1M2closeexpectation}, the truncated distribution have a small probability of being less than $0$ and the expectation of the truncated distribution is close to the original distribution. Hence,}
	&= \Pr\left\{\hat{M_1} + \hat{M_2} \ge \Phi^*({\bf y})\right\} \\
	\shortintertext{By Markov inequality,}
	&\le  \frac{\E \left\{\hat{M_1}+\hat{M_2}\right\}}{\Phi^*({\bf y})} = \frac{\E \left\{\hat{M_1}\right\}}{\Phi^*({\bf y})} + \frac{\E \left\{\hat{M_2}\right\}}{\Phi^*({\bf y})}\\
	\shortintertext{Since, $\hat{M_2}$ is a positive random variable,}
	&= \frac{\E \left\{\hat{M_1}\right\}}{\Phi^*({\bf y})} +  \frac{\int\limits_0^\infty \Pr (\hat{M_2} > z) dz}{\Phi^*({\bf y})}\\%+ 4\varepsilon\\
	%&= \frac{\E \left\{\hat{M_1}\right\}}{\Phi^*(y)} +  \frac{\int_0^\infty \Pr (M_2 > z) dz}{\Phi^*(y)}\\%+ 4\varepsilon\\
	&\le \frac{\E \left\{\hat{M_1}\right\}}{\Phi^*({\bf y})} + \frac{\int\limits_0^\infty \sum_{\underset{t \ge \tau, s\ne t}{s,t}}\Pr \left(\alpha\left(w_t - w_s\right)/\lambda_t > z\right) dz}{\Phi^*({\bf y})}\\%+ 4\varepsilon\\
	&= \frac{\E \left\{\hat{M_1}\right\}}{\Phi^*({\bf y})} + \frac{\int\limits_0^\infty \sum_{\underset{t \ge \tau, s\ne t}{s,t}} \exp \left(-z^2\lambda_t^2/4{\norm{\alpha}}^2\right) dz}{\Phi^*({\bf y})}\\% + 4\varepsilon
\end{align*}
Hence, the probability of Type-I error, is minimized by minimizing $\norm{\alpha}^2$. 
\end{proof}
\subsection{Proof of Lemma~\ref{lem:M1M2:alwayspositive}}
\label{appendix:proof:lem:M1M2:alwayspositive}
\vspace{-2em}
\begin{align*}
	\Pr \left\{M_1 \le 0\right\} 	&= \Pr \left\{\underset{\underset{s\ne t}{s,t}}{\max} \left({p}_t (w_t-w_s)\right) \le 0\right\} \\
	%\shortintertext{For any subset $\xi$, $\xi \subset \mathcal{M}$,}
	%&\le \Pr \left\{\underset{i}{\max} \xi_i \le 0\right\}
	\shortintertext{Choosing the set $\xi \subset \mathcal{M}$ as defined in~\eqref{eqn:defxi} and since the set $\xi$ are negative dependent from Lemma~\ref{lem:negativedep},}
	&\le \Pr \left\{\underset{i}{\max\;} \xi_i \le 0\right\} \le \underset{i}{\prod} \Pr  \left\{\xi_i \le 0\right\} %\label{eqn:NAXI}
	\shortintertext{Each of the term if $\xi_i = \left({p}_t (w_t-w_{t+1})\right)$ is distributed as $\mathcal{N}(0,2\norm{p_t}^2)$. Let $F_{\mathcal{N}(\mu,\sigma^2)}$ is the cdf of Gaussian random variable with mean $\mu$ and variance $\sigma^2$. Noting that $F_{\mathcal{N}(0,\sigma^2)}(0)=1/2$, we have the following}
	&=  \underset{t}{\prod} F_{\mathcal{N}(0,2\norm{p_t}^2)}(0) = \underset{t}{\prod} \frac{1}{2} = \frac{1}{2^{T}} < \varepsilon. 
	%\shortintertext{where $F_{\mathcal{N}(\mu,\sigma^2)}$ is the cdf of normal random variable with mean $\mu$ and variance $\sigma^2$. Noting that $F_{\mathcal{N}(0,\sigma^2)}(0)=1/2$, we have the following }
	%&= \underset{t}{\prod} \frac{1}{2} = \frac{1}{2^{T}} < \varepsilon 
\end{align*}
The proof for the second part is similar by an appropriate choice of a negative dependent set, $\xi$ and is hence omitted. \qed
\subsection{Proof of Lemma~\ref{lem:M1M2closeexpectation}}
\label{appendix:proof:lem:M1M2closeexpectation}
From the definition of the random variable $\hat{M_1}$ in~\eqref{eqn:m1:m2:hat}, 
\begin{align}
	%\E\left\{\hat{M_1}\right\} &= \int_{-\infty}^{+\infty} x f_{\hat{M_1}}(x) dx =  \int_{0}^{+\infty} x f_{{M_1}}(x) dx +  \Pr\left(M_1 < 0 \right) \nonumber \\
	\E\left\{\hat{M_1}\right\} &= \int\limits_{0}^{+\infty} x f_{{M_1}}(x) dx +  \Pr\left(M_1 < 0 \right) \nonumber\\
				   &<  \int\limits_{0}^{+\infty} x f_{{M_1}}(x) dx + \varepsilon, \label{eqn:M1hat:lb} 
\end{align}
where the inequality in~\eqref{eqn:M1hat:lb} follows from Lemma~\ref{lem:M1M2:alwayspositive}. The expectation of $M_1$ is given by
\begin{equation}
	%\shortintertext{By Lemma~\ref{lem:M1M2:alwayspositive}, }
	%&<  \int_{0}^{+\infty} x f_{{M_1}}(x) dx + \varepsilon. \label{eqn:M1hat:lb}\\
	%\shortintertext{The expectation of $M_1$ is given by}
	\E\left\{{M_1}\right\} = \E\left\{{M_1}\mathds{1}\left\{x \ge 0\right\}\right\} + \E\left\{{M_1}\mathds{1}\left\{x \le 0\right\}\right\}  \label{eqn:exp:M1}
	%&= \int_{0}^{+\infty} x f_{{M_1}}(x) dx + \int_{-\infty}^{0} x f_{{M_1}}(x) dx \nonumber 
\end{equation}
To continue with the proof, we derive a lower bound on $\E\left\{{M_1}\mathds{1}\left\{x \le 0\right\}\right\}$, the second term in~\eqref{eqn:exp:M1}. 

The following upper bound follows trivially,
\begin{equation}
	\E\left\{{M_1}\mathds{1}\left\{x \le 0\right\}\right\} \le 0. 
\end{equation}
For computing the lower bound, we proceed by integration by parts
\begin{equation*}
	%\begin{aligned}
	\E\left\{{M_1}\mathds{1}\left\{x \le 0\right\}\right\} = \int\limits_{-\infty}^{0} x f_{M_1}(x) dx  %\\
	%\ifcompress \else &= x F_{M_1}(x)|_{-\infty}^{0} - \int_{-\infty}^{0} F_{M_1}(x) dx \nonumber \\ \fi
	%&= - \int_{-\infty}^{0} F_{M_1}(x) dx \\
	%&
	= - \int\limits_{-\infty}^{0} \Pr \left\{M_1 \le x\right\}dx.
	%\end{aligned}
\end{equation*}
Choosing the negative dependent subset $\xi \subset \mathcal{M}$ defined in~\eqref{eqn:defxi}, and noting that each $\xi_i$ is distributed as $\mathcal{N}(0,2\norm{p_i}^2)$ and using analytical expression for bounds of the cdf of the Gaussian density, we obtain   
\begin{equation}
	%\shortintertext{For the negative dependent subset $\xi \subset \mathcal{M}$ defined in~\eqref{eqn:defxi},}
	\ge  - \int\limits_{-\infty}^{0} \prod_i \Pr \left\{\xi_i \le x \right\}dx. %\nonumber 
	%\shortintertext{Since cdf is monotone, we have,}
	%&\ge -\frac{1}{2^{T-1}} \int_{-\infty}^{0} \Pr \left\{\xi_k \le x \right\}dx. \nonumber
	%\shortintertext{Each $\xi_k$ is distributed as $\mathcal{N}(0,2\norm{p_k}^2)$ and the distribution is symmetric,}
	%&= -\frac{1}{2^{T-1}} \int^{+\infty}_{0} \Pr \left\{\xi_k \ge x \right\}dx \nonumber \\
	%\ifcompress \else &= -\frac{1}{2^{T-1}} \int^{+\infty}_{0} \text{Q}(\sqrt{2\norm{p_k}^2}x) dx \nonumber \\ \fi
	%\shortintertext{Using the analytical expression for integral of the Q-function and algebraic manipulation provides,}
	%&=  -\frac{1}{2^{T-1}} \int^{+\infty}_{0} \frac{1}{2} \text{erfc}(\sqrt{2\norm{p_k}^2}x/\sqrt{2}) dx \\
	%&=  -\frac{1}{2^{T}} \int^{+\infty}_{0} \text{erfc}(\norm{p_k}x) dx \\
	%&=  -\frac{1}{2^{T}} \left[x \text{ erfc}(\norm{p_k}x) - \frac{\me^{-\norm{p_k}^2 x^2}}{\sqrt{\pi}\norm{p_k}}\right]_{0}^{+\infty} \\
	%&=  -\frac{1}{2^{T}} \frac{1}{\sqrt{\pi}\norm{p_k}} \nonumber \\
	\ge -\varepsilon \label{eqn:M1:neg:lb} 
\end{equation}
From~\cref{eqn:M1hat:lb,eqn:M1:neg:lb} we get the first part of the Lemma~\ref{lem:M1M2closeexpectation}.The proof for the second part is similar and hence omitted.\qed
\subsection{CUSUM algorithm for Utility change point detection}
\label{appendix:cusum:utility:change:point}
\begin{minipage}{\columnwidth}
\begin{algorithm}[H]
\caption{CUSUM algorithm for Utility change point detection}
\label{algo:cusum:algorithm}
\begin{algorithmic}[1]
	\Initialize{Set threshold $\rho>0$. \\ Set cumulative sum $S(0) = 0$. \\ Set decision function $D(0) = 0$. }
	\For{$t = 1 \text{ to }T$}
	\\ For probe $p_t$ and observed response $y_t$, 
	\State  $x_t(0) =  \underset{\left\{p_t^\prime x \le I_t\right\}}{\argmax\;} v(x)$, with $v(x)$ as in~\eqref{eqn:simulation:cobb:douglas}. 
	\State  $x_t(1) =  \underset{\left\{p_t^\prime x \le I_t\right\}}{\argmax\;} v(x)+\alpha^\prime x$. 
	\State Likelihood $\ell(y_t,i) = \prod_{n=1}^m \Pr(y_t^n|x_t^n(i))\footnote{In our example, the probability is given by the Gaussian distribution. }\; ;i = 0,1$. 
	\State Instantaneous log likelihood $s(t) = \log(\frac{\ell(y_t,1)}{\ell(y_t,0)})$. 
	\State $S(t) = S(t-1) + s(t)$. 
	\State \scalebox{0.9}{$G(t) = \left\{G(t-1)+ s(t)\right\}^+$}, where \scalebox{0.9}{$\left\{x \right\}^+ = \max\left\{x,0\right\}$}. 
	\If{$G(t) > \rho$} 
	\State Change Point Estimate $\hat{\tau} = \underset{1\le \tau\le t}{\argmin \;} S(\tau-1)$
	\State break
	\EndIf
	\EndFor
\end{algorithmic}
\end{algorithm}
\end{minipage}

\bibliographystyle{\paperbibstyle}


\end{document}